\documentclass[reqno]{amsart}

\usepackage[usenames, dvipsnames]{color}
\usepackage{graphicx}
\usepackage{verbatim}

\usepackage{amsmath}
\usepackage{amsfonts}
\usepackage{amsmath}
\usepackage{amssymb}
\usepackage{amsthm}

\renewcommand{\epsilon}{\varepsilon}
\def\bR{\mathbb{R}}
\def\bN{\mathbb{N}}

\def\cA{\mathcal{A}}

\newcommand{\dist}{\operatorname{dist}}
\newcommand{\diam}{\operatorname{diam}}

\makeatletter
\def\dashint{\operatorname%
{\,\,\text{\bf--}\kern-.98em\DOTSI\intop\ilimits@\!\!}}
\makeatother

\def\polhk#1{\setbox0=\hbox{#1}{\ooalign{\hidewidth
   \lower1.5ex\hbox{`}\hidewidth\crcr\unhbox0}}}

\newtheorem{theorem}{Theorem}[section]
\newtheorem{thm}[theorem]{Theorem}
\newtheorem{pro}[theorem]{Proposition}

\newtheorem{lem}[theorem]{Lemma}
\newtheorem{lemma}[theorem]{Lemma}
\newtheorem{cor}[theorem]{Corollary}

\theoremstyle{definition}
\newtheorem{defn}[theorem]{Definition}
\theoremstyle{remark}
\newtheorem{rem}[theorem]{Remark}

\newtheorem{ex}[theorem]{Example}

\newtheorem{assumption}[theorem]{Assumption}

\begin{document}

\title[Parabolic operators in a time varying domain]{Boundary value problems for parabolic operators in a time-varying domain}

\author[S. Cho]{Sungwon Cho}
\address[S. Cho]{Department of Mathematics Education,
Gwangju national university of education,
93 Pilmunlo Bukgu,
Gwangju 500-703, Republic of Korea}
\email{scho@gnue.ac.kr}
\thanks{S. Cho was supported by Basic Science Research Program through the National Research Foundation of Korea (NRF) funded by the Ministry of Science, ICT \& Future Planning (2012-0003253).
}

\author[H.Dong]{Hongjie Dong}
\address[H. Dong]{Division of Applied Mathematics, Brown University,
182 George Street, Providence, RI 02912, USA}
\email{Hongjie\_Dong@brown.edu}
\thanks{H. Dong was partially supported by the NSF under agreement DMS-1056737.}

\author[D. Kim]{Doyoon Kim}
\address[D. Kim]{Department of Applied Mathematics, Kyung Hee University,
1732 Deogyeong-daero, Giheung-gu, Yongin-si, Gyeonggi-do 446-701, Republic of Korea}
\email{doyoonkim@khu.ac.kr}
\thanks{
D. Kim was supported by Basic Science Research Program through the National Research Foundation of Korea (NRF) funded by the Ministry of Science, ICT \& Future Planning (2011-0013960).
}



\begin{abstract}
We prove  the existence of unique solutions to the Dirichlet boundary value problems for linear second-order uniformly parabolic  operators in either divergence or  non-divergence form with boundary blowup low-order coefficients.
The domain is possibly time varying, non-smooth, and satisfies the exterior measure condition.
\end{abstract}

\maketitle

\section{Introduction} \label{introduction} 

In this paper we consider parabolic operators in divergence form
\begin{equation*} \label{D}\tag{D}
Lu =   D_t (u) -  D_j (a_{ij}
D_i  u) +  b_i D_i u - D_i (c_i u) + c_0 u
\end{equation*}
and in non-divergence form
\begin{equation*} \label{ND}\tag{ND}
Lu= D_t u  - a_{ij} D_{ij}u  + b_i D_i u +c_0 u
\end{equation*}
in a time-varying domain $Q$ in $\mathbb{R}^{n+1}$, $n\ge 1$, with boundary blowup low-order coefficients. Here and in the sequel,
$$
D_i := \frac{\partial }{\partial x_i}, \quad D_{ij} := D_i D_j, \,\,i ,j =1, \ldots , n,\quad D_t :=\frac{\partial}{\partial t },
$$
some derivatives in parentheses in divergence form are understood in  the weak sense, and summation over repeated indices is assumed. For convenience of notation, in the sequel we set $c_i=0,i=1,\ldots,n,$ in the non-divergence case.

With the  operator $L$ in \eqref{D} or \eqref{ND}, we study the following boundary value problems of Dirichlet type:
\begin{equation*}
			\label{DP}\tag{DP}
\left\{	
\begin{aligned}
L u &= f \quad \text{ in } Q,
\\
u &=g \quad \text{ on } \partial_p Q,
\end{aligned}
\right.
\end{equation*}
where
$f = f_0  -D_i f_i$ for the divergence case and $\partial_p Q$ is the parabolic boundary of $Q$ (see Definition \ref{parabolic_boundary} below). We prove that there exist unique solutions to the Dirichlet problems \eqref{DP}
when the domain satisfies the exterior measure condition and the boundary data is zero ($g\equiv0$).
In the non-divergence case, solutions satisfy the equation in the strong sense, and are locally in $W_p^{1,2}$, $p > (n+2)/2$.
While in the divergence case, they are understood in the weak sense. In both cases, solutions are continuous up to the boundary.
The coefficients which we consider have two features. First, concerning the leading coefficients, while in the divergence case we do not impose any regularity assumptions, in the non-divergence case we assume that
they have vanishing mean oscillations (VMO) with respect to the spacial variables and merely measurable with respect to the time variable.
Second, the lower-order coefficients may blow up near the boundary with a certain optimal growth condition.

Indeed, there is an extensive literature on the existence of solutions to the boundary value problem \eqref{DP} in a straight cylindrical domain with lower-order coefficients which are bounded or in certain Sobolev spaces.
See, for instance, \cite{LSU, La, Kr87, BC93, Li, Kr08} and the references therein.

Regarding  non-divergence form parabolic equations in time-varying domains (or more general degenerate elliptic-parabolic equations) one may find related results in Fichera \cite{Fichera1, Fichera2}, Oleinik \cite{OA1, OA3}, Kohn--Nirenberg \cite{KN}, and Krylov \cite{Kr02}, where, under certain assumptions, the existence, uniqueness, and regularity of solutions were discussed.
The solutions to parabolic equations in non-divergence form considered here are called $L_p$-strong solutions in \cite{Crandall}, where the authors treated various types of solutions to nonlinear equations.
We note that in \cite{Crandall} Crandall, Kocan, and \'Swi{\polhk{e}}ch considered equations in a cylindrical domain satisfying a uniform exterior cone condition
and the $L_p$-strong solutions are locally in $W_p^{1,2}$, $p > p_0$, where $p_0>(n+2)/2$ is a number close to $n+1$.
We also mention that in \cite{Li08}  Lieberman treated a similar problem for  a non-divergence elliptic operator in a cylindrical domain with blowup lower-order coefficients in weighed H\"older spaces.

As noted above, in the non-divergence case we assume that the leading coefficients are in a class of VMO functions. The study of elliptic and parabolic equations with VMO coefficients was initiated by Chiarenza, Frasca, and Longo \cite{CFL91}, and continued in \cite{CFL93} and \cite{BC93}. The class of leading coefficients in this paper was introduced by Krylov \cite{Kr07} in the context of parabolic equations in the whole space. See also \cite{Kim08, DY10, Dg08, DK09} for further development, the results of which we shall use in the proof.

For divergence form equations in time-varying domains, Yong \cite{Y89} proved the unique existence of weak solutions by using a penalization method when the domain satisfies the exterior measure condition (Definition \ref{A}) and its cross section at time $t$ is simply connected. He considered equations with a non-zero initial condition, and coefficients and the data $f_0, f_i$ are in  some suitable Lebesgue spaces, so that the weak solutions are actually H\"older continuous up to the boundary.
Later in \cite{BHL} Brown  et al. obtained a similar solvability result in a parabolic Lipschitz time-varying  domain with bounded measurable coefficients and more general square integrable data. One may refer to Lions \cite{Li61} for existence results in an abstract framework. Recently, Byun and Wang \cite{BW} obtained certain $L_p$-estimates for equations in time-varying $\delta$-Reifenberg domains. We also refer the reader to \cite{Li86, LM95, HL96,RN03, Ny08} and the references therein for other results about boundary value problems in time-varying domains. Of course, the boundary value problem in curvilinear cylinders can be deduced from the estimates in ``straight cylinders" by using a change of variables as long as the domains are sufficiently regular.

For the Laplace operator, we recall that a necessary and sufficient condition for the solvability of the corresponding boundary value problem to \eqref{DP} is the celebrated Wiener's criterion. See,  for example, \cite{W, KE, LK}.
As to the heat equation, an analogous result was established by Evans and Gariepy \cite{EG82}.  We are going to use this result in our proofs below.


To formulate our main results, we introduce some notation, function spaces, and assumptions. A typical point in $\bR^{n+1}$ is denoted by $X=(x,t)$, where $x\in \bR^n$ and $t\in \bR$.
The \emph{parabolic distance} between points
$X=(x,t)$ and $Y=(y,s)$ in $\mathbb R \sp{n+1}$
is $$
|X-Y|:=\max\{|x-y|, |t-s|^{\frac{1}{2}} \}.
$$
For any $Y=(y,s)\in \mathbb R \sp{n+1}$ and $r>0$, we set
$$
B_{r}(y):=\{x \in \mathbb R \sp n : |x-y|<r \}
$$
and
\begin{equation*}
C_{r}(Y):=B_{r}(y)\times (s-r^{2},s)=\{X=(x,t) \in \mathbb R  \sp{n+1}: |X-Y|<r,\,\, t<s\}
\end{equation*}
to be a ball in $\bR^n$ and a {\em standard parabolic cylinder} in $\bR^{n+1}$, respectively. We also set
$\hat C_r(Y)=B_r(y)\times (s-r^2,s+r^2)$.
Let $|\Gamma |:=|\Gamma |_{n+1}$
be the $n+1$-dimensional Lebesgue measure of a set $\Gamma$ in $\bR^{n+1}$.
For any real number $c$, denote $c_{+}:=\max \left( c,0\right)$ and $c_{-}:=\max \left( -c,0\right)$.

We assume that the coefficients $a_{ij}$ are defined on $\mathbb{R}^{n+1}$ and satisfy  the following uniform ellipticity condition:
there exists  a constant $\nu \in (0,1]$ such that
\begin{align*} \label{UE}\tag{UE}
 \nu |\xi|^2 \le \sum_{i,j=1}^n a_{ij} (X) \xi_i \xi_j ,
\quad
\sup_{i,j}   |a_{ij}(X)| \le \nu^{-1}
\end{align*}
for all $X\in\bR^{n+1}$ and $\xi =(\xi_1, ... \xi_n) \in \mathbb{R}^n$. In the non-divergence case, we impose the following {\em vanishing mean oscillation} (VMO) condition  on $a_{ij}$ with respect to $x$. We denote
\begin{equation*}
\omega_a(R):=\sup_{r\in (0,R]}\sup_{(x_0,t_0)\in \bR^{d+1}}\sup_{i,j}
\dashint_{C_r(t_0,x_0)} | a_{ij}(x,t) - \dashint_{B_r(x_0)} a_{ij}(y,t)\,dy | \, dX,
\end{equation*}
where $\displaystyle \dashint_C f(Y)\,d Y$ is the average of $f$ over $C$.

\begin{assumption}	\label{RA}
We have
$$
\omega_a(R)\to 0 \quad \text{ as } R \to 0^+.
$$
 \end{assumption}
Note that, under this assumption, no regularity is reguired for $a_{ij}$ as functions of $t$. For instance, $\omega_a(R) = 0$ if $a_{ij} = a_{ij}(t)$.
In the non-divergence case, without loss of generality, we may assume $a_{ij}=a_{ji}$. However, we do not impose such condition in the divergence case.

For lower-order coefficients, we assume the following:
\begin{equation}
					\label{D-sign_01}
 \int_Q  \left(c_0(X) \phi(X) + c_i (X) D_i \phi (X)\right) dX \ge 0 ,
\quad \forall \phi \ge 0, \quad \phi \in C^\infty_0 (Q)
\end{equation}
in the divergence case, and
\begin{align*} 
c_0 (X) \ge 0 \quad \text{ in } Q
\end{align*}
in the non-divergence case.
Note that the above two conditions can be collectively referred to the following unified condition:
\begin{align} \label{sign}
 L 1 \ge 0,
\end{align}
which implies the maximum principle for $L$. The lower-order coefficients $b_i$ and $c_i$ are allowed to blow up near the boundary under a certain growth condition, stated in the theorem below. In light of the example before the proof of Theorem \ref{main}, this growth condition is optimal.

We impose the following exterior measure condition (or condition (A)) on the domain.
\begin{defn} \label{A}
An open set $Q \subset \mathbb R  \sp {n+1}$ satisfies the condition (A)
if there exists a constant $\theta _{0}\in (0,1)$
such that for any $X=(x,t) \in \partial_p Q $ and $r>0$,
we have  $|C_{r}(X)\setminus Q |>\theta _{0}|C_{r}|$.
\end{defn}

We deal simultaneously with both cases of divergence \eqref{D} and non-divergence  \eqref{ND} form.
In fact, the first author and Safonov took a unified approach and obtained global a priori H\"older estimates in \cite[Corollary 3.6, Theorem 3.10]{CS}
for elliptic equations without lower-order terms and with locally bounded right-hand side, and in \cite[Theorem 3.4, Theorem 4.2]{Ch10} for the parabolic case.
For this approach, it is convenient to introduce the solution space $W(Q)$,
which varies according to \eqref{D} and \eqref{ND}.
Let $p_0 \in (\frac {n+2} 2,\infty) $ be a fixed constant.
We use the notation:
\begin{eqnarray*}
&W(Q)&:=W^{ND} (Q):=W^{2,1}_{p_0, \text{loc}} (Q)\cap C(\bar{Q}) \quad \text{in the case \eqref{ND}, } \\
&W(Q)&:=W^{D} (Q):= H (Q)\cap C(\bar{Q})\quad \text{in the case \eqref{D},}
\end{eqnarray*}
where $\bar Q:=Q\cup \partial Q$ and
\begin{multline*}
H(Q) = \{ u \in L_2 (Q) \,|\,  D_i u \in L_{2, \text{loc}} (Q),\\
D_t (u)  = g_0 +  D_i g_i  \text{ for some } g_0 \in L_{1, \text{loc}} (Q) ,  g_i \in  L_{2, \text{loc}} (Q)\}.
\end{multline*}
Here
$  f \in L_{p,\text{loc}} (Q)$, $p>0$ if and only if $f  \in  L_{p}(Q')$
for any open set $Q' \Subset_p Q$. Throughout the paper we use $Q' \Subset_p Q$ to indicate that $Q'\subset Q$ and $\dist (Q' , \partial_p Q ) > 0$.

In both cases, the functions $u\in W$ are continuous on $\bar{Q}$.
In addition, in the non-divergence case \eqref{ND},
the functions $u$ have strong derivatives $D_{i}u$, $D_{ij}u$, $D_t u$ in the Lebesgue space $L_{p_0,\text{loc}}(Q)$.
In this case, the relations $Lu=f$ or $Lu\leq f$ are understood in the almost everywhere sense in $Q$.
In the divergence case \eqref{D}, the functions $u\in W$ have weak (generalized) derivatives
$D_{i}u$ and $D_t u$, and $Lu= (\leq, \geq ) f_0 -D_i f_i$ for $f_0, f_i  \in L_{1,\text{loc}}(Q)$ is understood in the following weak sense:
\begin{multline*} 
\int_Q \left( - u D_t \phi  + a_{ij}   D_i u  D_j \phi+
   b_i D_i u \phi +  c_i u D_i \phi +c_0 u \phi - f_0 \phi  - f_i D_i \phi \right)  \, d X\\= (\leq ,\,\geq )   0
\end{multline*}
for any nonnegative function $ \phi \in C_0^{\infty }(Q)$.

Regarding the data, we consider more general function spaces than those in \cite{CS} and \cite{Ch10}.
We first define, for $\beta \in (0,2)$, $p>0$,
\[
\|f\|_{F_{\beta,p}(Q)} :=
\sup_{Y_0 \in \partial_p Q, r>0  }
 \left( \dashint_{\hat C_r(Y_0) \cap Q } |d^{2-\beta} (X)  f(X) |^{p} \, d X \right)^{1/p},
\]
where $d(X) :=\dist(X, \partial_p Q )$,
and we say $f \in F_{\beta,p}(Q)$ when $\|f\|_{F_{\beta,p}(Q)} <\infty$.
Note that if $f \in F_{\beta,p}(Q)$, then $f \in L_{p,\text{loc}} ( Q )$.
For some $\beta\in (0,1)$,
we set
\begin{align*}
F(Q)& = F_\beta (Q)\\
 &:= \{ f= (f_0,f_1,\ldots,f_n )  \,:\,f_0\in F_{\beta,p_0}(Q),f_i\in F_{1+\beta,2p_0}(Q) ,i=1,\ldots,n\},\\
 \|f\|_{F(Q)}&:=\|f_0\|_{F_{\beta,p_0}(Q)}+\sum_{i=1}^n \|f_i\|_{F_{1+\beta,2p_0}(Q)}
\end{align*}
for the  divergence case,
and $F(Q) = F_\beta (Q) := F_{\beta,p_0}(Q) $ for the non-divergence case.

Now we are ready to state the main results of this paper:

\begin{thm} \label{main}
Let $p_0\in (\frac {n+2} 2,\infty)$, $L$ be a uniformly parabolic operator in either divergence \eqref{D}
or non-divergence form \eqref{ND}
satisfying \eqref{sign},
and $Q$ be a bounded domain in $\mathbb{R}^{n+1}$ satisfying the measure condition (A) with a constant $\theta_0 \in (0,1)$.
Suppose that $c_0 \in L_{\infty,\text{loc}}(Q)
$ and
\begin{align*}
|b_i|, |c_i|=o(d^{-1}),\quad d = d(X),
\end{align*}
i.e., there exists a nondecreasing function $\gamma$ on $\bar\bR_+$ such that $\gamma(d)\to 0$ as $d\to 0$ and
$$
|b_i|, |c_i|\le d^{-1}\gamma(d).
$$
For the non-divergence case, we further assume that the coefficients $a_{ij}$ satisfy Assumption \ref{RA}. Let $\beta_1:= \beta_1 (n, \nu , \theta_0)$
be  the constant from Proposition~\ref{UF} below.

Then, for any $f \in F_\beta  (Q), \beta\in (0,\beta_1)$,
there exists a unique solution $u \in W(Q)$ to the Dirichlet problem \eqref{DP} when $g\equiv 0$.
\end{thm}

It is worth noting that from the proofs below the solution $u$ is globally H\"older continuous in $\bar Q$ and, in the divergence case, $D_t(u)=g_0+D_i g_i$ for some $g_0,g_i\in L_{2,\text{loc}}(Q)$.
The corresponding results for elliptic operators are also obtained by following the proofs in this paper.

Here we illustrate the idea in the proof of Theorem \ref{main}. Our proof relies on the growth lemma, from which we deduce an a priori uniform boundary estimate in Proposition~\ref{UF}. To prove the existence result, in the non-divergence case,  first we approximate the operator $L$ by a sequence of operators $L^{k}$ which become the heat operator near the boundary and coincide with the original operator $L$ in the interior of the domain. We then find a sequence of solutions $u_k$ corresponding to the operators $L^k$ by Perron's method, which requires barrier functions and the solvability in cylindrical domains. We construct certain barrier functions by using the result of Evans and Gariepy \cite{EG82} mentioned above and an idea in Krylov \cite{Kr89}. Under Assumption \ref{RA}, the $W^{1,2}_p$ solvability of non-divergence form parabolic equations in cylindrical domains is also available in the literature.
By using the a priori boundary estimate and the interior $W^{1,2}_p$ estimate, we are able to show that along a subsequence $u_k$ converge locally uniformly to a solution $u\in W(Q)$ of the original equation. The divergence case is a bit more involved. We additionally take mollifications of the coefficients and data, and rewrite the approximating equations into non-divergence form equations, for which the solvability has already been proved. We then show the convergence of a subsequence of $u_k$ to a solution $u\in W(Q)$ of the original equation by again using the a priori boundary estimate and the interior De Giorgi--Nash--Moser estimate.

The remaining part of the paper is organized as follows. We present several auxiliary results in the next section including a version of the maximum principle for solutions in $W(Q)$. Section \ref{sec3} is devoted to a growth lemma (Lemma \ref{growth}) and a pointwise estimate (Lemma \ref{pointwise}). In Section \ref{sec4}, we obtain an a priori boundary estimate which is crucial in our argument, and in Section \ref{proofs} we complete the proof of the main results. In the Appendix, we show that any domains satisfying the exterior measure condition also satisfy Wiener's criterion, which is used in the construction of the barrier functions.

\section{Auxiliary results} \label{aux}
This section is devoted to some auxiliary results. First we recall the following standard definition.



\begin{defn} \label{parabolic_boundary}
Let $Q$ be an open set in $\mathbb R  \sp{n+1}$.
The \emph{parabolic boundary}  $\partial_p Q$ of $Q$ is the set of all points
$X_{0}=\left(x_{0},t_{0}\right) \in \partial Q$
such that there exists a continuous function $x=x\left( t\right) $ on an interval $[t_{0},t_{0}+\delta )$
with values in $\mathbb R \sp n$ satisfying $\ x\left( t_{0}\right) =x_{0}$
and $\left( x\left( t\right) ,t\right) \in Q$ for all $t\in \left(t_{0},t_{0}+\delta \right)$.
Here $x=x\left( t\right) $ and $\delta >0$ depend on $X_{0}$.
\end{defn}

We denote $\bar \partial_p Q$ to be the closure of $\partial_p Q$ in $\partial Q$. By the continuity, it is easily seen that the condition (A) is satisfied for any $X\in \bar\partial_p Q$.
The next lemma follows from Lemma 2.3 of \cite{Ch10}, which reads that $\partial Q\setminus \bar\partial_p Q$ is locally flat.

\begin{lemma}
                    \label{lem11}
Let $Q$ be an open set in $\bR^{n+1}$ and $X_0=(x_0,t_0)\in \partial Q\setminus \bar\partial_p Q$. Then there exists $r>0$ such that $B_r(x_0)\times \{t=t_0\}\subset \partial Q\setminus \bar\partial_p Q$, and
$$
B_r(x_0)\times (t_0,t_0+r^2)\subset \bR^{n+1}\setminus \bar Q,\quad B_r(x_0)\times (t_0-r^2,t_0)\subset Q.
$$
\end{lemma}

For any interior point $X\in Q$, we have the following measure condition for sufficiently large $r$.

\begin{lem}
                \label{lem2.0}
Assume that $Q$ satisfies the measure condition (A) with a constant $\theta_0 \in (0,1)$. Let $X=(x,t)\in Q$ and denote $\rho=d(X)$. Then for any $r\ge 4\rho/\theta_0$, we have
\begin{equation}
                                    \label{eq2.56}
|C_r(X)\setminus Q|>\theta_0 2^{-n-2}|C_r|.
\end{equation}
\end{lem}
\begin{proof}
By a scaling argument, without loss of generality we may assume that $\rho=1$. Let $Y_0=(y,s)$ be a point on $\bar\partial_p Q$ such that $d(X,Y_0)=1$. In the case when $s\le t$, we have $C_{r/2}(Y_0)\subset C_r(X)$. By using the condition (A) at $Y_0$,
\begin{equation}
                                        \label{eq3.22}
|C_r(X)\setminus Q|\ge |C_{r/2}(Y_0)\setminus Q|>\theta_0|C_{r/2}|=\theta_0 2^{-n-2}|C_r|,
\end{equation}
which gives \eqref{eq2.56}.

Next we consider the case when $s>t$.
We claim that $C_{r/2}(X)\setminus Q$ is not empty. Otherwise, we would have $C_{r/2}(X)\subset Q$. Since $\rho=1$ and $r\ge 4/\theta_0$,
$$
|C_{r/4}(Y_0)\setminus Q|\le |C_{r/4}(Y_0)\setminus C_{r/2}(X)|\le \theta_0|C_{r/4}|,
$$
which contradicts with the condition (A) at $Y_0$. Now we fix a point $Y_1\in C_{r/2}(X)\setminus Q$. Denote $\{Y_\tau\,|\,Y_\tau=(1-\tau)X+\tau Y_1,\tau\in [0,1]\}$ be the line segment connecting $X$ and $Y_1$. Let $\tau^*$ be the smallest number in $(0,1)$ such that $Y_{\tau^*}\in \bR^{d+1} \setminus Q$. Clearly, we have $Y_{\tau^*}\in \partial_p Q$ and $C_{r/2}(Y_{\tau^*})\subset C_r(X)$. By using the condition (A) at $Y_{\tau^*}$, we obtain \eqref{eq3.22} with $Y_{\tau^*}$ in place of $Y_0$. The lemma is proved.
\end{proof}

In the remaining part of this section, we do not impose the condition (A) on $Q$. The following lemma is useful in approximating $u_{+}$ by smooth functions.

\begin{lemma} \label{lemmaF}
Let $G \in C^{\infty }(\mathbb R )$ and $ u \in W(Q)$.
Then $ v: =G(u)\in W(Q)$. In addition, assume $G',G'' \geq 0$ on $\mathbb R$ and $G(0)=0$.
 Then, for a function $f$ defined in $Q$ such that $f\in L_{p_0,\text{loc}}(Q)$ in the case \eqref{ND},
or $f=f_0-D_if_i$, $f_0 \in L_{1,\text{loc}}(Q)$, $f_i \in L_{2,\text{loc}}(Q)$ in the case \eqref{D},
satisfying $ Lu\leq f$,
we have $L v \leq F$ in $Q$ where
$$
\left.
\begin{aligned}
&F:=F_0-D_iF_i,
\\
&F_0 = G ' (u)f_0 - G '' (u) a_{ij} D_iuD_j u +G''(u)f_iD_iu \in L_{1,\text{loc}}(Q),
\\
&F_i = G'(u)f_i\in L_{2,\text{loc}}(Q),
\end{aligned}
\right\}
\quad
\text{for } \eqref{D}
$$
and
\begin{equation*} 
F:=G ' (u)f - G '' (u)a_{ij} D_iuD_ju \quad \text{ for } \eqref{ND} .
\end{equation*}
In particular, we have $L v\leq 0$ in $ Q$ provided that $L u\leq 0$ in $Q$.
\end{lemma}
\begin{proof}
Clearly, in both cases $v\in C(\bar Q)$. First we consider the non-divergence case. We have
$$
D_i v=G'(u)D_i u,\quad v_t=G'(u)u_t,
\quad D_{ij}v=G'(u)D_{ij}u+G''(u)D_iuD_j u.
$$
Since $G'(u), G''(u)\in C(\bar Q)$, we get $D_i v,v_t\in L_{p_0,\text{loc}}(Q)$.
By the parabolic Sobolev embedding theorem (see, for instance, \cite[Chap. II]{LSU}), $D_i u\in L_{p,\text{loc}}(Q)$ with
$$
p=p_0 (n+2)/(n+2-p_0)>2p_0,
$$
which implies that $D_{ij}v\in L_{p_0,\text{loc}}(Q)$. Therefore, $v\in W(Q)$.
Since $G'\ge 0$ and $c_0\ge 0$, we get
$$
Lv=LG(u)=G'(u)Lu-G''(u)a_{ij} D_iuD_j u +c_0(G(u)-G'(u)u)\le F\quad\text{in}\,\,Q,
$$
where we have used the simple inequality
\begin{equation}
                                \label{eq1.57}
G(u)\le G'(u)u
\end{equation}
because $G(0)=0$ and $G$ is convex.

In the divergence case, by the definition of the space $H(Q)$, we have $D_i v=G'(u)D_i u\in L_{2,\text{loc}}(Q)$ and
$$
D_t (v) =G'(u)u_t=G'(u)(g_0+D_ig_i)
$$
provided that  $ D_t (u)   = g_0 + D_i g_i$ for some $g_0\in L_{1,\text{loc}}(Q)$ and $g_i\in L_{2,\text{loc}}(Q)$.
Then
$$
D_t (v)=\tilde g_0+D_i\tilde g_i,
$$
where
$$
\tilde g_0=G'(u)g_0-G''(u) g_iD_iu,\quad \tilde g_i=G'(u)g_i.
$$
It is easily seen that $\tilde g_0\in L_{1,\text{loc}}(Q)$ and $\tilde g_i\in L_{2,\text{loc}}(Q)$. Therefore, $v\in W(Q)$. To show the desired inequality, it suffices to prove that for any positive $\phi\in C_0^\infty(Q)$,
\begin{multline*}
\int_Q  - G(u) \phi_t + G'(u)\big(a_{ij} D_i u D_j \phi+ b_i  D_i u \phi\big)+c_iG(u)D_i\phi\\
+c_0G(u)\phi-F_0 \phi - F_i D_i \phi \, d X\le 0.
\end{multline*}
From \eqref{D-sign_01} and \eqref{eq1.57}, we have
\begin{multline*}
\int_Q c_iG(u)D_i\phi+c_0G(u)\phi\,dX=\int_Q c_iD_i(G(u)\phi)-c_iG'(u)\phi D_i u+c_0G(u)\phi\,dX\\
\le \int_Q c_iD_i(G'(u)u\phi)-c_i G'(u)\phi D_i u+c_0G'(u)u\phi\,dX\\
=\int_Q c_iuD_i(G'(u)\phi)+c_0G'(u)u\phi\,dX.
\end{multline*}
Therefore, we only need to show that
\begin{equation}
                                    \label{eq4.32}
\int_Q   -G(u) \phi_t + a_{ij}D_iuD_j\psi+ b_i \psi D_i u +c_i uD_i\psi+c_0u\psi -  f_0\psi - f_iD_i\psi \, d X\le 0,
\end{equation}
where $\psi:=G'(u) \phi$. We use a standard mollification argument.  Define $\psi^\varepsilon=G'(u^{(\varepsilon)}) \phi$, where $u^{(\varepsilon)}$ is the standard mollification of $u$. By the definition of a weak solution, \begin{equation}
							\label{eq0807_1}
\int_Q   -u D_t \psi^{\varepsilon} + a_{ij}D_iuD_j\psi^{\varepsilon}+ b_i \psi^{\varepsilon} D_i u +c_i uD_i\psi^{\varepsilon}+c_0u\psi^{\varepsilon} -  f_0\psi^{\varepsilon} - f_iD_i\psi^{\varepsilon} \, d X\le 0.
\end{equation}
Let $Q'=\text{supp} \psi \Subset_p Q$. Since $u^{(\varepsilon)}\to u$ in $C(\bar Q')$ and $Du^{(\varepsilon)}\to Du$ in $L_2(Q')$, we have
$$
\psi^\varepsilon\to \psi\quad\text{in}\,\,C(\bar Q'),\quad
D\psi^\varepsilon\to D\psi\quad \text{in}\,\,L_2(Q').
$$
From this together with $u\in W(Q)$, we see that the left-hand side of \eqref{eq0807_1}
converges to that of \eqref{eq4.32} as $\varepsilon \to 0$.
In particular,
\begin{align*}
&- \int_Q u D_t \psi^{\varepsilon} \, dX
= \int_Q g_0 \psi^{\varepsilon} \, dX - \int_Q g_i D_i \psi^{\varepsilon} \, dX\\
&= \int_Q \left( G'(u^{(\varepsilon)}) g_0  -  G''(u^{(\varepsilon)}) g_i D_i u^{(\varepsilon)} \right) \phi \, dX
- \int_Q  G'(u^{(\varepsilon)}) g_i D_i \phi \, dX\\
&\to \int_Q \left(\tilde{g}_0 \phi - \tilde{g}_i D_i \phi\right) \, dX
= - \int_Q G(u) \phi_t \, dX.
\end{align*}
This completes the proof of \eqref{eq4.32}. The second assertion follows from the first one by taking $f=0$ and using $G''\ge 0$ and the ellipticity condition.
The lemma is proved.
\end{proof}

The following lemma allows us to reduce our consideration to functions
defined on a standard cylinder $C_{r}(X_{0})$
rather than on a general open set $Q\subset \mathbb R \sp{n+1}$. For operators without lower-order terms, a similar result is claimed in Theorem 2.6 of \cite{Ch10}, the proof of which, however, contains a flaw.

\begin{lemma} \label{ext}
Let $Q$ be an open set in $\mathbb R \sp{n+1}$ and $u \in W(Q)$ satisfy $Lu\leq 0$ in $Q$ with an operator $L$
in the form \eqref{ND} or \eqref{D}.
Suppose $u\leq 0$ on $\bar C_r  \cap \bar\partial_p  Q$,
where $C_r:=C_r(X_{0})$, $X_{0}\in Q$.
Then for any $\varepsilon >0$, there exists a function $u_{\varepsilon}\in W(C_{r})$ which vanishes in a neighborhood of $\bar C_r \cap \partial Q$ and satisfies
\begin{equation*}
u_{\varepsilon }\geq 0, \;
Lu_{\varepsilon }\leq 0\quad \text{in}\,\,C_r,\quad
u_{\varepsilon }\equiv 0\quad \text{in}\,\,C_{r}\setminus Q,
\end{equation*}
and
\begin{equation*} 
(u-\varepsilon )_{+}(X_{0})\leq u_{\varepsilon }(X_{0}),
\quad
u_\varepsilon  \le u_+  \quad \text{in} \,\, Q.
\end{equation*}
\end{lemma}

\begin{proof}
The idea of the proof is to modify $u$ so that it vanishes near $\bar C_r\cap \partial Q$.
For $\varepsilon>0$, we choose a convex non-decreasing nonnegative function $G_\varepsilon\in C^\infty(\bR)$ such that $(s-\varepsilon)_+\le G_\varepsilon(s)\le (s-\varepsilon/2)_+$ on $\bR$.
We first modify $u$ near $\bar C_r\cap \bar\partial_p Q$. By Lemma \ref{lemmaF}, we have $v_\varepsilon:=G_\varepsilon(u)\in W(Q)$ and it satisfies
$$
v_\varepsilon\ge 0,\quad Lv_\varepsilon\le 0,\quad (u-\varepsilon)_+\le v_\varepsilon\le (u-\varepsilon/2)_+\quad\text{in}\,\, Q.
$$
Since $u\le 0$ on $\bar C_r\cap \bar\partial_p Q$, $v_\varepsilon$ vanishes in a neighborhood of $\bar C_r\cap \bar\partial_p Q$ in $\bar Q$. Now we modify $u$ near $\bar C_r\cap \big(\partial Q\setminus \bar\partial_p Q\big)$. Thanks to Lemma \ref{lem11},
$$
\partial Q\setminus \bar\partial_p Q=\bigcup_{\alpha\in \cA}S_\alpha\times \{t=t_\alpha\},
$$
where $\cA$ is an index set. Here, for each $\alpha\in \cA$, $S_\alpha$ is an open set in $\bR^n$, $t_\alpha\in \bR$, and
$$
\partial S_\alpha\times \{t=t_\alpha\}\subset \bar\partial_p Q.
$$
In fact, $\cA$ is at most countable (see Remark \ref{rem1119}). However, we will not use this in the proof below.
Since $u$ is uniformly continuous in $\bar Q$ and $u\le 0$ on $\bar C_r\cap \bar\partial_p Q$, we can find $\delta>0$ such that $u\le \varepsilon/2$ in $\{X\in \bar C_r\cap \bar Q\,|\,\dist(X,\bar C_r\cap \bar \partial_p Q)<\delta\}$. Clearly, the set
$$
\{X\in \bar C_r\cap \partial Q\,|\,\dist(X,\bar C_r\cap \bar \partial_p Q)\ge \delta/2\}\subset \partial Q\setminus \bar \partial_p Q
$$ is compact, which has a finite covering by $S_{\alpha_k}\times \{t=t_{\alpha_k}\}$, where $\alpha_1,\ldots,\alpha_M\in \cA$ and $M\in \bN$. Denote
$$
S^\delta_{\alpha}=\big\{X\in S_\alpha\times \{t=t_\alpha\}\cap \bar C_r\,|\,\dist(X,\bar C_r\cap \bar \partial_p Q)\ge \delta/2\big\}.
$$
By Lemma \ref{lem11}, there is a small constant $\delta_1\in (0,\delta/2)$ such that
$$
S^\delta_{\alpha_k}\times [t_{\alpha_k}-\delta_1^2,t_{\alpha_k})\subset \bar C_r \cap Q,\quad k=1,\ldots,M,
$$
and these sets do not intersect each other. We choose a smooth function $\eta=\eta(x,t)$ in $\bar C_r\cap Q$ satisfying the following three properties:\\
(i) $0\le \eta\le 1$ in $\bar C_r\cap Q$; \\
(ii) For each $k=1,\dots,M$, in $S^\delta_{\alpha_k}\times (t_{\alpha_k}-\delta_1^2,t_{\alpha_k})$ the function $\eta$ is independent of $x$, non-increasing in $t$, and $\eta=0$ in $S^\delta_{\alpha_k}\times (t_{\alpha_k}-\delta_1^2/2,t_{\alpha_k})$;\\
(iii) $\eta=1$ in $\bar C_r\cap Q\setminus \bigcup_{k=1}^M \big(S_{\alpha_k}\times (t_{\alpha_k}-\delta_1^2,t_{\alpha_k})\big)$, which contains $X_0$.\\
Observe that $D_t\eta\le 0$ and $D\eta=0$ in
$$
\{X\in \bar C_r\cap  Q\,|\,\dist(X,\bar C_r\cap \bar \partial_p Q) \ge \delta\}\supset \{X\in \bar C_r\cap Q\,|\,v_\varepsilon(X)> 0\}.
$$
Consequently, $u_\varepsilon:=v_\varepsilon \eta$ satisfies $Lu_\varepsilon=\eta Lv_\varepsilon+v_\varepsilon D_t \eta\le 0$ in $\bar C_r\cap Q$. Noting that $u_\varepsilon$ vanishes in a neighborhood of $\bar C_r\cap \partial Q$ in $\bar Q$, we can extend $u_\varepsilon$ to be zero in $C_r\setminus Q$. It is now straightforward to check that $u_\varepsilon$ satisfies all the properties in the lemma.
\end{proof}

\begin{rem}
							\label{rem1119}
One example of space-time domains with infinitely many flat portions of the non-parabolic boundary can be obtained by connecting a sequence of shrinking cubes by triangular prisms as in Figure \ref{fig01} infinitely many times.
Note that $\partial Q\setminus \bar\partial_p Q$ is a countable union of the top surfaces of these cubes and the domain satisfies the exterior measure condition.
\begin{figure}[h]
\begin{center}
\includegraphics[scale=0.35]{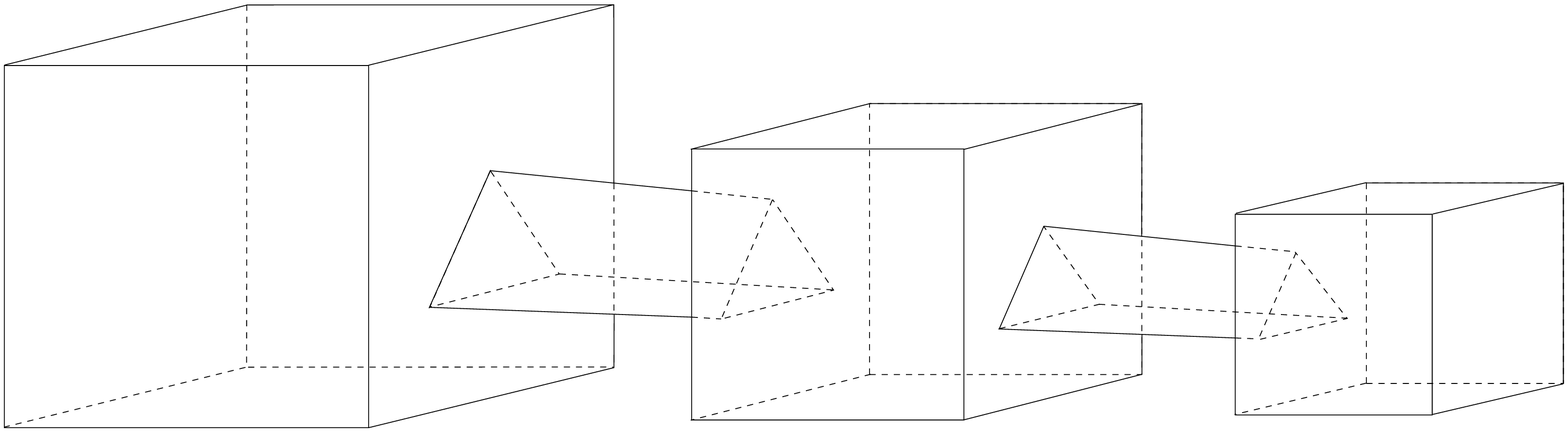}
\end{center}
\caption{}\label{fig01}
\end{figure}
\end{rem}

Finally, we prove a version of the maximum principle for solutions in $W(Q)$.

\begin{lem}[Maximum principle] \label{max}
Let $u \in W(Q )$ and
 $Lu \le 0$ in $Q$, where $L$ is   a uniformly parabolic operator   in either divergence \eqref{D}
or non-divergence form \eqref{ND}
satisfying \eqref{sign} with locally bounded lower order coefficients (i.e., bounded on sets $Q' \Subset_p Q$).
For the non-divergence case, we further assume that the coefficients $a_{ij}$ satisfy Assumption \ref{RA}.
Then
\begin{equation}
                            \label{eq11.45}
\sup_Q u
\le
\sup_{\partial_p Q}
u\vee 0.
\end{equation}
Similarly, if $Lu \ge 0$, then
$$
\inf_Q u \ge
\inf_{\partial_p Q}
u \wedge 0.
$$
\end{lem}

\begin{proof}
In the divergence case, this is classical. See, for instance, \cite[Sec. 6.7]{Li}.

Next we treat the non-divergence case.
It suffices to prove \eqref{eq11.45}. Due to \eqref{sign}, we may assume $\sup_{\partial_p Q}u\le 0$.
We first consider the special case when $Q=\Omega\times (0,T)$ is a cylindrical domain, $\Omega$ is a bounded $C^{1,1}$ domain in $\bR^n$, and the coefficients are all bounded. In this case, similar estimates can be found in \cite[Theorem 11.8.1]{Kr08} for elliptic operators with continuous coefficients.
The proof there is based on the $W^2_p$ estimate for elliptic operators.
Since the $W^{2,1}_p$ estimate for parabolic operators in cylindrical domains with VMO coefficients is available (cf. \cite{DY10} and \cite{DK09}), by the same argument one can derive the maximum principle for $L$. It is standard to extend to the general case by a contradiction argument. Suppose that $\sup_Q u>0$. Since $Q$ is bounded, we may assume that $Q\subset \{t>t_0\}$ for some $t_0\in \bR$. Then for $\delta>0$ sufficiently small, $v:=u-\delta(t-t_0)$ attains its maximum $M>0$ at a point $X_0\in \bar Q\setminus \bar\partial_p Q$ and satisfies $Lv\le -\delta$ in $Q$. Take a small $r>0$ such that  $C_r(X_0)\Subset_p Q$ and a smooth function $\eta$ such that $\eta(X_0)=1$ and $\eta=0$ on $\partial_p C_r(X_0)$. It is easily seen that for sufficiently small $\varepsilon>0$, we have
$$
L(v+\varepsilon \eta)\le 0\quad \text{in}\quad C_r(X_0),\quad
v+\varepsilon \eta\le M \quad\text{on}\quad \partial_p C_r(X_0).
$$
By the maximum principle in cylindrical domains proved above, we have
$$
M+\varepsilon=v(X_0)+\varepsilon \eta(X_0)\le M;
$$
a contradiction. The lemma is proved.
\end{proof}

\begin{cor}[Comparison principle] \label{comp}
Let $u,v \in W(Q)$ and,
under the same assumptions  on $L$ as in Lemma \ref{max},
$Lu \ge Lv$ in $Q$,
$u \ge v $ on $\partial_p Q$.
Then $u \ge v $ in $Q$.
\end{cor}
\begin{proof}
We apply Lemma \ref{max} to $u-v$.
\end{proof}

By the comparison principle, it is immediate to see that a solution to \eqref{DP} is unique
if it exists.

\section{Growth lemma and pointwise estimates}   \label{sec3}

The following \emph{growth lemma} is
in the spirit of the book by Landis \cite[Section 1.4]{La}.
In one form or another, the growth lemma was used in the proofs of Harnack inequalities for solutions to elliptic and parabolic equations. See, for instance, \cite{Sa} and the references therein.

\begin{lem}[Growth Lemma]  \label{growth}
Let $r_0\in (0,\infty)$ be a constant and $Q \subset \mathbb R \sp{n+1}$ be a bounded domain. Suppose $X_{0}=(x_{0},t_{0})\in \bar Q \setminus \bar\partial_p Q$ and
\begin{equation*} 
|C_{r}\setminus Q|>\theta |C_{r}|,\quad 0<\theta <1, \quad C_{r}:=C_{r}(X_{0}) \quad
\text{for some }
r\in (0,r_0] .
\end{equation*}
Also, let $L$  be the uniformly parabolic  operator in  \eqref{D} or \eqref{ND}
satisfying \eqref{sign}
with $|b_i|, |c_i|\le K_0$.

Then for any function $u\in W( Q_{r})$,  $Q_{r}:=C_r \cap Q$, satisfying
\begin{equation*} 
Lu \leq 0 \, \text{ in }Q_{r}  \quad \text{and} \quad u\leq 0 \,\, \text{on} \,\,
C_r \cap  \partial_p   Q,
\end{equation*}
we have
\begin{equation}\label{GL}
u(X_0) \leq \beta_0 \, \sup_{C_{r}}u_+\end{equation}
with constant
\begin{equation*}
\beta_0 =\beta_0 (n,\nu ,\theta, K_0, r_0 ) \in (0,1).
\end{equation*}
Assume that $u$ is extended as $u\equiv 0$ on
$C_{r}\setminus Q$, so that the right-hand side of \eqref{GL} is
always well defined.
\end{lem}
\begin{proof}
First we make some reductions. By considering slightly smaller cylinders $C_{r-\varepsilon}(x_0,t_0-\varepsilon^2)$ and letting $\varepsilon\to 0$, without loss of generality we may assume that
$u\le 0$ on $\bar C_r\cap \bar \partial_p Q$ and $X_0 \in Q$. In particular, if $X_0 \in \partial Q \setminus \bar\partial_p Q$, then by Lemma \ref{lem11} $(x_0, t_0-\varepsilon^2) \in Q$ with a sufficiently small $\varepsilon>0$.
Next by a scaling and a translation, we further assume $X_0=(0,0)$ and $r=1$. Note that the dependence of $\beta_0$ on $r_0$ comes from the scaling argument.
Thanks to Lemma \ref{ext} with $Q= Q_1$,
there exists $u_\varepsilon$ such that
$$
(u-\varepsilon)_+ (X_0) \le u_\varepsilon (X_0),
\quad u_\varepsilon  \le u_+  \quad \text{ in } Q_1
$$
$$
u_\varepsilon\in W(C_1),\quad
u_\varepsilon \ge 0,\quad
Lu_\varepsilon\le 0\,\;\text{in}\,\,C_1,\quad
u_\varepsilon\equiv 0\,\;\text{in}\,\,C_1\setminus Q.
$$
We claim that $u_\varepsilon (X_0) \le \beta_0 \, \sup_{C_1} (u_\varepsilon)_+ $.
Once this is proved, using $ u_\varepsilon \le u_+$ in $Q_1$ and $ u_\varepsilon \equiv 0$ in $C_1 \setminus Q_1$,
we obtain
$\sup_{C_1} (u_\varepsilon)_+  \le \sup_{C_1}  u_+$.
Taking  $\varepsilon \to 0^+$, we will eventually  get \eqref{GL}.
Now
$u_\varepsilon$  will be denoted by $u$ for the rest of proof.
Moreover, since $u\ge 0$ in $C_1$ and $L1\ge 0$, by considering
$$
L-c_0+D_ic_i=D_t-D_j(a_{ij}D_i)+(b_i-c_i)D_i
$$
instead of $L$, we only need to treat the case $c_0=c_i=0$.

We first treat the case when $u$ and the coefficients are smooth. In this case, for operators without the lower-order terms, the lemma was proved in \cite{FeS} under a slightly weaker condition. The general case follows by using an idea in Remark 6.2 there. Indeed, we introduce a new variable $x_0\in \bR$ and define
$$
v(x_0,x,t)=e^{x_0+At}\big(\sup_{C_1}u-u(x,t)\big),\quad
\tilde C_r=(-r,r)\times C_r,\quad \tilde Q=\bR\times Q.$$
Set
$$
a_{0i}=0,\quad a_{i0}=-b_i,\quad i=1,\ldots,n,\quad a_{00}=A
$$
in the divergence case, and
$$
a_{0i}=a_{i0}=-b_i/2,\quad i=1,\ldots, n,\quad a_{00}=A
$$
in the non-divergence case. Here we choose $A=A(n,\nu,K_0)>0$ sufficiently large such that the $(n+1)\times (n+1)$ matrix $(a_{ij})_{i,j=0}^n$ satisfies \eqref{UE} with a possibly smaller constant $\nu$. Then it is easy to check that $v$ satisfies
$$
L^0 v=v_t-\sum_{i,j=0}^n D_j(a_{ij}D_i v)\ge 0
$$
in the divergence case, and
$$
L^0 v=v_t-\sum_{i,j=0}^n a_{ij}D_{ij} v\ge 0
$$
in the non-divergence case.
Now \cite[Corollary 5.4]{FeS} is applicable to $v$. Indeed, we can find a small constant $\delta_0\in (0,1/4)$ depending only on $n$ and $\theta$ such that
$$
2\delta_0^2+(1-\delta_0)^2\le 1,\quad |\tilde C_{1-\delta_0}(0,-s)\setminus \tilde Q|>\theta|\tilde C_1|/2,
$$ where $s=1-(1-\delta_0)^2$.
Moreover, since $u=0$ in $C_1\setminus Q$,
we have $v\ge e^{-1-A}   \sup_{C_1}u$ in $\tilde C_{1-\delta_0}(0,-s)\setminus \tilde Q$.
Thus we
get
$$
v\ge \beta e^{-1-A}\sup_{C_{1}}u \quad \text{in}\,\, \tilde C_{\delta_0}
$$ for some $\beta=\beta(n,\nu,\theta,K_0)\in (0,1)$, which implies \eqref{GL} with $\beta_0=1-\beta e^{-1-A}$.
We note that the proof above only requires
\begin{equation}
                    \label{eq10.32}
Lu\le 0\,\;\text{in}\,\,C_1,\quad
u\le 0\,\;\text{in}\,\,C_1\setminus Q.
\end{equation}

Next we remove the smoothness assumption on $u$ and the coefficients by using an approximation argument. We take a small $\delta$ depending on $n$ and $\theta$ such that $|C_{1-\delta}\setminus Q|\ge \theta|C_{1-\delta}|/2$. For $\varepsilon\in (0,\delta)$, let $u^{(\varepsilon)}$, $a^{(\varepsilon)}_{ij}$, and $b^{(\varepsilon)}_i$ be the standard mollification of $u$, $a_{ij}$, and $b_i$.

By the Lebesgue lemma,
$$
(a^{(\varepsilon)}_{ij},b_i^{(\varepsilon)})\to (a_{ij},b_i)\,\;\text{a.e. in}\,\,C_{1-\delta}.
$$

{\em Non-divergence case:} Clearly, $Lu\in L_{p_0,\text{loc}}(C_1)$. Since $a^{(\varepsilon)}_{ij}$, $b_i^{(\varepsilon)}$, and $u^{(\varepsilon)}$ are smooth, the Dirichlet problem
$$
L^\varepsilon v^\varepsilon:=v^\varepsilon_t-a_{ij}^{(\varepsilon)} D_{ij} v^\varepsilon+b_i^{(\varepsilon)} D_i v^\varepsilon=(Lu)^{(\varepsilon)} \quad \text{in}\,\,C_{1-\delta}
$$
with the boundary condition $v^\varepsilon=u^{(\varepsilon)}$ on $\partial_p C_{1-\delta}$ has a unique smooth solution $v^\varepsilon$ in $C_{1-\delta}$. Moreover,
$$
L^\varepsilon(v^\varepsilon-u^{(\varepsilon)})=(Lu)^{(\varepsilon)}-Lu+Lu-L^\varepsilon u^{(\varepsilon)}\to 0\quad\text{in}\,\,L_{p_0}(C_{1-\delta}).
$$
By Lemma \ref{pointwise} below and the uniform continuity of $u$, for any $h>0$ we have $|v^\varepsilon-u|\le h$ in $C_{1-\delta}$ provided that $\varepsilon$ is sufficiently small. Because
$$
L^\varepsilon (v^\varepsilon-h)=(Lu)^{(\varepsilon)}\le 0 \quad  \text{ in } C_{1-\delta}
$$
and
$$
v^\varepsilon-h\le u=0 \quad  \text{ in } C_{1-\delta}\setminus Q,
$$
\eqref{eq10.32} is satisfied with $v^\varepsilon-h$ in place of $u$ and $C_{1-\delta}$ in place of $C_1$.
Since $v^\varepsilon-h$, $a^{(\varepsilon)}_{ij}$, and $b^{(\varepsilon)}_i$ are smooth, by the proof above, we get
$$
u-2h\le v^\varepsilon -h\le \beta_0 \sup_{C_{1-\delta}}(v^\varepsilon -h)\le \beta_0\sup_{C_1}u_+\quad \text{in}\,\,C_{\delta_0(1-\delta)}.
$$
Letting $h\to 0$ and $\delta\to 0$ gives \eqref{GL}.

{\em Divergence case:}
Let
\[
D_t (u) - D_j ( a_{ij} D_i u ) + b_i D_i u  =: f_0 - D_i f_i    .
\]
By a property of convolution, we have
\[
D_t u^{(\varepsilon)} - D_j \left( ( a_{ij} D_i u )^{(\varepsilon)}  \right) + (b_i D_i u)^{(\varepsilon)}
 = f_0^{(\varepsilon)} - D_i f_i^{(\varepsilon)}    .
\]
Similar to the non-divergence case, let $v^\varepsilon$ be the weak solution of
\[
L^\varepsilon v^\varepsilon : =  D_t v^\varepsilon - D_j \left( a_{ij}^{ (\varepsilon) } D_i v^\varepsilon \right)
+ b_i^{(\varepsilon)}  D_i v^\varepsilon = f_0^{(\varepsilon)} - D_i f_i^{(\varepsilon)}
\quad  \text{ in } C_{1-\delta}
\]
with the boundary condition $v^\varepsilon = u^{(\varepsilon )} $ on
$\partial_p C_{1-\delta} $.
Note that
\[
L^\varepsilon u^{(\varepsilon )}  =  D_j \left( ( a_{ij} D_i u )^{(\varepsilon)} \right)
 - (b_i D_i u)^{(\varepsilon)} + f_0^{(\varepsilon)} - D_i f_i^{(\varepsilon)}
 - D_j \left( a_{ij}^{ (\varepsilon) } D_i u^{(\varepsilon)} \right)
+ b_i^{(\varepsilon)}  D_i u^{(\varepsilon )}.
\]
Thus,
\[
L^\varepsilon ( v^\varepsilon - u^{(\varepsilon)} ) =
D_j \left( a_{ij}^{ (\varepsilon) } D_i u^{(\varepsilon )} - (a_{ij} D_i u)^{(\varepsilon)}  \right) + (b_i D_i u )^{(\varepsilon)} - b_i^{(\varepsilon)} D_i u^{(\varepsilon)}    .
\]
It is easy to see that
\[
a_{ij}^{ (\varepsilon) } D_i u^{(\varepsilon )}  - (a_{ij} D_i u)^{(\varepsilon )} \to 0, \quad
 (b_i D_i u )^{(\varepsilon )}  - b_i^{(\varepsilon )} D_i u^{(\varepsilon )} \to 0
\]
in $L_2(C_{1-\delta})$.
Due to the energy inequality, we have $v^\varepsilon\to u$ in $L_2(C_{1-\delta})$, and, after extracting a subsequence, $v^{\varepsilon_k} \to u$ (a.e.) in $C_{1-\delta}$ for a decreasing sequence $\varepsilon_k \searrow 0$.
Fix a small $h>0$, denote
$$
S_{\varepsilon,h,1-\delta}:=\{v^\varepsilon>h,u\le 0\}\cap C_{1-\delta}.
$$
By Chebyshev's inequality, as $\varepsilon\to 0$,
$$
|S_{\varepsilon,h,1-\delta}|\le
 \frac{1}{h^2}
\int_{C_{1-\delta}}(v^\varepsilon-u)^2\,dX\to 0.
$$
Thus, for $\varepsilon$ sufficiently small, we have
$$
|\{v^\varepsilon\le h\}\cap C_{1-\delta}|
 \ge |\{u\le 0\}\cap C_{1-\delta}|-|S_{\varepsilon,h,1-\delta}|>\frac \theta 4 |C_{1-\delta}|.
$$
Since $v^\varepsilon$, $a^{(\varepsilon)}_{ij}$, and $b_i^{(\varepsilon)}$ are smooth, and $(Lu)^{(\varepsilon)}\le 0$ in $C_{1-\delta}$, from the proof above and the maximum principle, we reach
$$
\sup_{C_{\delta_0(1-\delta)}}(v^\varepsilon -h)\le
\beta_0
\sup_{C_{1-\delta}}(v^\varepsilon -h)\le
\beta_0
\sup_{\partial_p C_{1-\delta}}(v^\varepsilon -h)  =
\beta_0
\sup_{\partial_p C_{1-\delta}}u^{(\varepsilon)} -h
\le
\beta_0
\sup_{C_1}u.
$$
Bearing in mind that $v^{\varepsilon_k} \to u$ (a.e.) in $C_{1-\delta}$ as $k\to \infty$ and $u$ is continuous, we get
$$
\sup_{C_{\delta_0(1-\delta)}}(u -h)\le \beta_0
\sup_{C_1}u.
$$
Since $h$ is an arbitrary positive constant, we obtain the desired estimate.
\end{proof}

\begin{rem} \label{GLremark}
If we additionally assume $u=0$ on $\partial_p Q$, then under the conditions of Lemma \ref{growth}, from $u(X_0) \le \beta_0 \sup_{C_r} u_+ $ we can derive
\begin{equation*} 
\sup_{Q \cap \hat C_{\delta_0 r}(X_0)}u
\leq \beta_0 \, \sup_{Q \cap \hat C_{r}(X_0)} u_+\end{equation*}
with constant
\begin{equation*}
\delta_0=\delta_0(n,\theta)\in (0,1/4),   \
\beta_0 =\beta_0 (n,\nu ,\theta/2, K_0, r_0 ) \in (0,1).
\end{equation*}
Moreover, in this case $X_0$ can be any point on $\bar Q$.
For this, note that the measure condition
$$
|C_r (X) \setminus Q | >
\frac{\theta}{2} |C_r|
$$
holds  for any $X \in \hat C_{\delta_0 r}(X_0)$
with small $\delta_0$ depending on $n$ and $\theta$.
\end{rem}

In the sequel, $N=N(\cdots )$ denotes a constant depending only on the prescribed quantities,
such as $n$, $\nu$, etc., which are specified in the parentheses,  and the value of $N$ may change from line to line.

\begin{lemma}[Pointwise estimate] \label{pointwise}
Let $p_0\in (\frac {n+2} 2,\infty)$, $R_1 > 0$,
and $L$ be a uniformly parabolic operator (in the form \eqref{D} or \eqref{ND}) defined on a cylinder
$C_{R}:=C_{R}(X_{0}) \subset \mathbb{R}^{n+1}$, $R \le R_1$, which satisfies \eqref{sign} and $|b_i|, |c_i|, |c_0| \le K_0$.
In the non-divergence case we further assume that $a_{ij}$ satisfy Assumption \ref{RA}.
Then for any $f=f_0-D_if_i$, where $f_0 \in L_{p_0}(C_R)$ and  $f_i \in L_{2p_0}(C_R)$ for the divergence case,
and $f \in L_{p_0}(C_R)$ for the non-divergence case,
there exists a unique solution $w\in W(C_{R})$ to the equation
$$
L w = f  \text{ in  } C_{R}, \quad
w = 0  \text{ on  } \partial_p C_R.
$$
Moreover, $w$ is H\"older continuous in $C_R$ and
\begin{equation}   \label{15}
| w | \leq N_1 \left( \|f_0\|_{L_{p_0}(C_R) } + \| f_i \|_{L_{2p_0}(C_R)} \right) \; \text{ for the divergence case},
\end{equation}
where $N_{1}=N_1(n,\nu, K_0, R_1, p_0)$,
and
\begin{equation}   \label{16}
| w | \leq N_2  \|f\|_{L_{p_0}(C_R)}  \; \text{ for the non-divergence case},
\end{equation}
where $N_2=N_2(n,\nu, K_0, R_1, p_0, \omega_a)$.
\end{lemma}
\begin{proof}
In the divergence case, the unique existence of $w$ in $H(C_R)$ is classical by noting that $f_i\in L_2(C_R)$ and $f_0\in L_{2(n+2)/(n+4)}(C_R)$. The H\"older continuity and \eqref{15} are due to the parabolic De Giorgi--Nash--Moser estimate. See, for instance, \cite[Chap. 6]{Li}. In the non-divergence case, for operators without the lower-order terms, the unique solvability was first established in \cite{BC93}
under the assumption that the leading coefficients are VMO with respect to both $x$ and $t$. In the general case, the unique solvability in $W^{1,2}_{p_0}(C_R)$ can be found in \cite{DK09}; see Theorem 6 and Remark 1 there. Recall that $p_0>(n+2)/2$. The H\"older continuity and \eqref{16} then follow from the parabolic Sobolev embedding theorem.
The lemma is proved.
\end{proof}

We remark that if $p_0\ge n+1$, by the parabolic Alexandrov--Bakelman--Pucci estimate (cf. \cite{Kr76,T85}), the constant $N_2$ can be taken to be independent of the regularity of $a_{ij}$.

\section{Weighted uniform estimate}   \label{sec4} 

The a priori estimate in Proposition \ref{UF} below is a key ingredient in the proof of the existence of solutions.
The idea of the proof is based on \cite[Theorem 3.5]{CS}.
Our case is much more involved because the coefficients of the lower-order terms may blow up near the boundary. We use a re-scaling argument, for which we introduce some additional notation. For any $\rho>0$, we denote
$$
u^\rho (X) = u (\rho x,\rho^2 t) \quad \text{in}\,\,
\rho^{-1}Q:= \{(\rho^{-1}x,\rho^{-2}t) \,|\,  (x,t)  \in Q \}.
$$
For the divergence case \eqref{D}, by a simple scaling,
$$
L^\rho u^\rho := D_t ( u^\rho )-  D_j(a^\rho_{ij} D_i u^\rho) +  \rho b^\rho_i D_i u^\rho- \rho D_i (c^\rho_i u^\rho)
 + \rho^2 c^\rho_0 u^\rho = \rho^2 f^\rho_0-\rho D_if^\rho_i,
$$
where $a_{ij}^\rho (X)= a_{ij} (\rho x,\rho^2 t)$, etc.
For the non-divergence case \eqref{ND}, we have
\begin{align*} 
L^\rho u^\rho :=
 u_t^\rho -a_{ij}^\rho  D_{ij} u^\rho + \rho b_i^\rho D_i u^\rho+\rho^2 c_0^\rho u^\rho=\rho^2 f^\rho.
\end{align*}
Denote
$$
d_\rho (X) = \text{dist} (X, \partial_p(\rho^{-1} Q) ) = \rho^{-1} d( \rho x,\rho^2 t).
$$

Next we modify the operators near the parabolic boundary. For any $\rho>0$,
take a smooth non-negative function $\eta^\rho$ such that $\eta^\rho=0$ in $\{d< \rho/2\}$, $\eta^\rho=1$ in $\{d> \rho\}$, and its modulus of continuity $\omega_\eta(r)\le Nr/\rho$. For $\varepsilon\in (0,1)$, denote
$$
a^{\rho,\varepsilon}_{ij}=\eta^{\rho\varepsilon}(\rho x,\rho^2 t) a^\rho_{ij}+\big(1-\eta^{\rho\varepsilon}(\rho x,\rho^2 t)\big)\delta_{ij},
$$
which satisfy  \eqref{UE}, and the modulus of continuity $\omega_{a^{\rho,\varepsilon}}$ has an upper bound
$$
\omega_{a^{\rho,\varepsilon}}(r)\le N(\omega_a(\rho r)+r/\varepsilon).
$$
Note that this is independent of $\rho$ for $\rho\in (0,1)$.
Let
$$b_i^{\rho,\varepsilon}=\eta^{\rho\varepsilon}(\rho x,\rho^2 t)b^\rho_i,\quad c_i^{\rho,\varepsilon}=\eta^{\rho\varepsilon}(\rho x,\rho^2 t)c^\rho_i,
$$
$$
c_0^{\rho,\varepsilon}=\eta^{\rho\varepsilon}(\rho x,\rho^2 t)c^\rho_0+D_i\eta^{\rho\varepsilon}(\rho x,\rho^2 t)c^\rho_i,\quad f^{\rho,\varepsilon}=\eta^{\rho\varepsilon}(\rho x,\rho^2 t)f^\rho.
$$
Then $b_i^{\rho,\varepsilon},c_i^{\rho,\varepsilon}, c^{\rho,\varepsilon}_0 \in L_\infty(\bR^{n+1})$.
We define an operator $L^{\rho,\varepsilon}$ on $\bR^{n+1}$:
$$
L^{\rho,\varepsilon} w := D_t ( w )-  D_j(a^{\rho,\varepsilon}_{ij} D_i w) +  \rho b^{\rho,\varepsilon}_i D_i w- \rho D_i (c^{\rho,\varepsilon}_i w)
 + \rho^2 c^{\rho,\varepsilon}_0 w
$$
in the divergence case, and
$$
L^{\rho,\varepsilon} w :=
 w_t -a_{ij}^{\rho,\varepsilon}  D_{ij} w + \rho b_i^{\rho,\varepsilon} D_i w+\rho^2 c_0^{\rho,\varepsilon} w
$$
in the non-divergence case. It is easily seen that $L^{\rho,\varepsilon}1\ge 0$.

\begin{pro} \label{UF}
Let $p_0\in (\frac {n+2} 2,\infty)$ be a constant, $Q$ be a bounded domain in $\mathbb{R}^{n+1}$ satisfying the measure condition (A)
with a constant $\theta_0 >0$,
and $L$ be a uniformly parabolic operator in the form of \eqref{D} or \eqref{ND}
satisfying the same assumptions for the coefficients $a_{ij}$, $b_i$, $c_i$, and $c_0$ as in Theorem \ref{main}.
Then there exists a constant $\beta_1 = \beta_1(n,\nu,\theta_0) \in (0,1]$
such that the following hold true:

(i) For $\beta \in (0,\beta_1)$ and
 an open subset $Q' \Subset_p Q$, if  $u \in W(Q')$,  $f \in F(Q)$,  and
\[
Lu=f \text{ in } Q',\quad  u=0 \quad \text{ on } \partial_p Q',
\]
then
\[
\sup_{Q'} d^{-\beta} u \le N
\|f\|_{ F(Q) },
\]
where $d= d(X)$,
$N=N(n,\nu, \theta_0,\gamma, \diam(Q), \beta, p_0)$, which also depends on $\omega_a^*$ in the
non-divergence case.
Here
$$
\omega_a^*=\max\big\{\sup_{\rho\in (0,\rho_0)}\omega_{a^{\rho,\varepsilon}}, \,\, \omega_{a^{1,\rho_0/4}}\big\},
$$
where $\varepsilon\in (0,1)$ depending only on $n$, $\nu$, $\theta_0$, and $\beta$, and $\rho_0\in (0,1)$ depending only on the same parameters and $\gamma$.

(ii) The same estimate holds true for $Q'=Q$ if in addition we assume that
$f$ vanishes near $\partial_p Q$, and $b_i$ and $c_i$ are bounded (for instance, by $1$) near $\partial_p Q$.
\end{pro}
\begin{proof}
We first prove the assertion (i).
We assume that the set $Q'\cap \{u>0\}$ is
nonempty; otherwise there is nothing to prove. The assumption ${\rm
dist}\,(Q',\partial_p Q)>0$ allows us to claim
that $d^{-\beta }u\in C(\bar{Q'})$, and hence there
is a point $X_0\in \bar{Q'}$ at which
\begin{equation} \label{MU}
d^{-\beta } (X_0)  u(X_{0}) =
M:=\sup_{Q'}d^{-\beta }u>0.
\end{equation}
Set $\rho:=d(X_{0}):={\rm dist}\,(X_{0},\partial_p Q )>0$,
and choose a point $Y_{0}\in \bar\partial_p Q$ for which $d(X_{0})=|X_{0}-Y_{0}|$.
Without loss of generality, we assume $X_0=(0,0)$ and $Y_0=(y_0,s_0)$.

Now we set $R:=4/\theta_0$ and
$\beta_1 :=- \log_{2R} \beta_0 >0$, where
$$
\beta_0=\beta_0(n,\nu,\theta_0/2^{n+2},1,4/\theta_0)
$$
is the constant from  Lemma~\ref{growth}.
Since $0 < \beta < \beta_1$ and $R > 1$, we have
$$
1 - (R+1)^\beta \beta_0 >  1 - (R+1)^{\beta-\beta_1}> 0.$$
Now we take $\varepsilon\in (0,1)$, depending only on $n$, $\nu$, $\theta_0$, and $\beta$,  such that
\begin{equation}
							\label{eq0728}
\varepsilon^\beta + (R+1)^\beta \beta_0 < 1.
\end{equation}

Denote $\rho^{-1}Q':= \{(\rho^{-1}x,\rho^{-2}t) \,|\,  (x,t)  \in Q' \}$ and $Q_r=C_r\cap (\rho^{-1}Q)$.
Observe that
$$
|\rho b_i^\rho (X) |, \, |\rho c_i^\rho (X) | \le \rho\, o \big(d^{-1} (\rho x,\rho^2 t )\big) =
 \rho\, o ( \rho^{-1} d^{-1}_\rho (x,t) ) \le \varepsilon^{-1} \gamma\left((R+1)\rho\right)
$$
in  $\{X\in Q_R\,|\,d_\rho(X) \ge \varepsilon\}$, where we used $Q_R \subset \{X\in \rho^{-1}Q\,|\,d_\rho(X)<R+1\}$.

We choose $\rho_0\in (0,1)$ depending only on $n$, $\nu$, $\theta_0$,  $\beta$, and $\gamma$ such that
$$
\varepsilon^{-1} \gamma\left((R+1)\rho_0\right) \le 1.
$$
Next, we consider two cases: (i) $\rho < \rho_0$ and (ii) $\rho \ge \rho_0$.

{\em Case 1: $\rho< \rho_0$.} We note that $f^{\rho,\varepsilon}  \in L_{p_0} (C_R)$ in the non-divergence case, and $f_0^{\rho,\varepsilon}\in L_{p_0} (C_R)$, $f_i^{\rho,\varepsilon}\in L_{2p_0} (C_R)$ in the divergence case.
By Lemma~\ref{pointwise}, there exists a unique solution $w \in W(C_R)$
to $L^{\rho,\varepsilon} w = \rho^2 f^{\rho,\varepsilon}$ in the non-divergence case (or $L^{\rho,\varepsilon} w = \rho^2 f_0^{\rho,\varepsilon}-\rho D_i(f_i^{\rho,\varepsilon})$ in the divergence case) with the zero boundary condition $w=0$ on $\partial_p C_R$.
Moreover, $w$ satisfies  the property \eqref{15} (or \eqref{16}).

Consider the function
\begin{equation}
v^\rho (X):=u^\rho (X)- (\varepsilon\rho)^{\beta}M - w (X)-\sup_{C_R}|w|
\quad\text{on}\quad \rho^{-1} Q'\cap C_R,  \label{vu}
\end{equation}
which may be extended continuously outside $\rho^{-1}Q' \cap C_R$,
and define
\begin{equation*}
V:=\rho^{-1} Q'\cap \{ d_\rho > \varepsilon \}\cap \{v^\rho >0\}.
\end{equation*}
We see that
\[
M = \sup_{X \in Q'} d^{-\beta} (X) u(X)
= \rho^{-\beta} \sup_{ X  \in \rho^{-1} Q'} d_\rho^{-\beta} (X) u^\rho (X),
\]
which implies that
\[
u^\rho = d_\rho^\beta d_\rho^{-\beta} u^\rho \le d_\rho^\beta \rho^\beta M \le (\varepsilon \rho)^\beta M ,
\quad \text{ when } d_\rho \le \varepsilon.
\]

We first assume $v^\rho (0)>0$.
By Lemma \ref{lem2.0} applied to $\rho^{-1}Q$, we have
$$
|C_R \setminus V| \ge |C_R \setminus \rho^{-1}Q| \ge 2^{-n-2}\theta_0|C_R|.
$$
From $u^\rho=0$ on $\partial_p({\rho}^{-1} Q') $ and
$u^\rho \leq (\varepsilon\rho)^\beta M $ on $\{d_\rho \leq \varepsilon \}$,
it follows $v^\rho=0$ on $C_R \cap \partial_p V$.
Moreover, $0 \in \bar{V} \setminus \bar\partial_p V$ because $0\in \overline{\rho^{-1}Q'}\cap \{ d_\rho > \varepsilon \}\cap \{v^\rho >0\}$ and $0\notin \bar\partial_p \left(\rho^{-1} Q'\right)$.
We also observe that
$L^{\rho,\varepsilon} v^\rho \le 0$
in $V$, where we use the condition \eqref{sign}, the definition of $w$, and fact that $L^{\rho,\varepsilon}$ coincides with $L^\rho$ in $\{d_\rho\ge \varepsilon\}\cap \rho^{-1}Q$.
Then by Lemma ~\ref{growth} with $Q = V$,
we obtain
\begin{equation*}
v^\rho (0)
\leq \beta_0 \, \sup_{V\cap C_R}v^\rho
\leq \beta_0 \, \sup_{V\cap C_R}u^\rho,
\end{equation*}
where $\beta_0 =\beta_0 (n,\nu ,\theta _{0}/2^{n+2},1, R)\in (0,1)$.
Since $d_\rho(X) \le R+1$ on $C_R$, we see that
$u^\rho \leq \left( (R+1) \rho\right)^\beta M$
on $V\cap C_R$, which implies $v^\rho (0)\leq  \beta_0 \left((R+1)\rho\right)^{\beta}  M$.
Of course, the last estimate also
holds in the case $v^\rho(0)\leq 0$. From this estimate,
together with \eqref{MU} and \eqref{vu}, it follows that
\begin{equation*}
\rho^{\beta} M = u(0)= u^\rho(0)
\leq \rho^{\beta } \left[ \varepsilon^{\beta }+ (R+1)^\beta\beta_0\right] M
+ 2\sup_{C_R}|w|.
\end{equation*}
By the property \eqref{15} or \eqref{16} of the function $w$ on $C_R$
we have
$$
|w|\leq N_{1}
\|\rho^2 \tilde f_0^\rho\|_{L_{p_0} (C_R)}+N_1\sum_{i=1}^n\|\rho \tilde f_i^\rho\|_{L_{2p_0} (C_R)}
$$
in the divergence case, where
$N_1=N_1(n,\nu, \theta_0, p_0)>0$, and
\begin{equation*}
|w| \leq N_1
\|\rho^2 \tilde f^\rho\|_{L_{p_0} (C_R)}
\end{equation*}
in the non-divergence case, where
$N_1=N_1(n,\nu, \theta_0, p_0, \omega_{a^{\rho,\varepsilon}})>0$.

A direct computation shows that
\begin{align*}
\|\rho^2 \tilde f^\rho\|_{L_{p_0} (C_R)} &= \rho^2 \left( \int_{ Q_{R} \cap \{ d_\rho \ge \varepsilon\}} |f^\rho (Y) |^{p_0} \, dY \right)^{1/{p_0}}
\\
&=  \rho^2 \left( \rho^{-n-2} \int_{ Q_{R\rho} \cap \{ d \ge \varepsilon\rho \}} |f (X) |^{p_0} \, dX \right)^{1/{p_0}}
\\
&\le  \rho^{2-(n+2)/p_0 +\beta -2} \varepsilon^{\beta-2}\left( \int_{ Q_{R \rho} \cap \{ d \ge \varepsilon\rho \}}  | d^{2-\beta} (X)f (X) |^{p_0} \, dX \right)^{1/p_0}
\\
& \le  \rho^{\beta -(n+2)/p_0  } \varepsilon^{\beta-2}\|f\|_{F_{\beta, {p_0}}(Q)} \,
( 2|C_{(R+1)\rho}|)^{1/p_0}.
\end{align*}
Similarly, in the divergence case,
\begin{align*}
\|\rho^2 f_0^\rho\|_{L_{p_0} (Q_2)}&\le \rho^{\beta -(n+2)/p_0  }
\varepsilon^{\beta-2}\|f_0\|_{F_{\beta, {p_0}}(Q)} \,  (2|C_{(R+1)\rho}|)^{1/p_0} ,\\
\|\rho f_i^\rho\|_{L_{2p_0} (Q_2)}&\le \rho^{\beta -(n+2)/(2p_0)  }
\varepsilon^{\beta-1}\|f_i\|_{F_{\beta+1, {2p_0}}(Q)} \,  (2|C_{(R+1)\rho}|)^{1/(2p_0)}.
\end{align*}
Therefore,
\begin{equation*} 
M \leq \left[  \varepsilon^{\beta }+ (R+1)^{\beta } \beta_0 \right] M
+ N_1 \varepsilon^{\beta-2 }
\|f\|_{F(Q)}.
\end{equation*}
From this inequality and
\eqref{eq0728}, 
it follows that $M\leq N
\|f\|_{F_{\beta, q}(Q)}$, where $N=N(n,\nu, \theta_0, \beta, p_0)$, which also depends on $\omega_{a^{\rho,\varepsilon}}$ in the non-divergence case.

{\em Case 2: $\rho\ge \rho_0$.} In this case, it suffices to bound $|u(0)|$ in terms of $f$.
Let $R_0=\diam(Q)$ and
$$
Q^\rho=Q'\cap \{X\in Q\,:\,\dist(X,\partial_p Q)>  \rho/4\}.
$$
In the non-divergence case, let $w$ be the unique solution in $W(C_{R_0})$ of
$$
L^{1,\rho_0/4}w=f^{1,\rho_0/4}\,\,  \text{in}\,\, C_{R_0}, \quad
w = 0\,\,\text{on}\,\,\partial_p C_{R_0}.
$$
By Lemma \ref{pointwise}, we have
\begin{equation}
							\label{eq1110_1}
|w|\le N\|f\|_{L_{p_0}(Q^\rho)},
\end{equation}
where $N=N(n,\nu, \theta_0, \gamma, \beta, p_0,\diam(Q),\omega_{a^{1,\rho_0/4}})$.

Because $L^{1,\rho_0/4}$ coincides with $L$ in $Q^\rho$,
$$
L\Big(\sup_{\partial_p Q^\rho}u^+ +w+N\|f\|_{L_{p_0}(Q^\rho)}\Big)\ge f\quad \text{in}\,\,Q^\rho.
$$
Moreover,
$$
\sup_{\partial_p Q^\rho}u^+ +w+N\|f\|_{L_{p_0}(Q^\rho)} \ge u
\quad\text{on}\,\,\partial_p Q^\rho.
$$
By the comparison principle Corollary \ref{comp} and \eqref{eq1110_1}, we have
$$
u(0)\le \sup_{\partial_p Q^\rho}u^+ + N\|f\|_{L_{ p_0} (Q^\rho)}.
$$
Since $u = 0$ on $\partial_p Q'$, we observe that
$$
\sup_{\partial_p Q^\rho}u^+\le (\rho/4)^\beta M.
$$
Thus,
\begin{align*}
M = \rho^{-\beta}u(0)&\le 4^{-\beta}M+N\rho^{-\beta}\|f\|_{L_{p_0} (Q^\rho) }\\
&\le 4^{-\beta}M+
N\rho^{-\beta}
\left(\frac{\rho}{4}\right)^{\beta-2}
\|f \|_{F(Q)}
| Q |^{1/p_0}\\
&\le  4^{-\beta}M+ N\rho_0^{-2} \|f \|_{F(Q)},
\end{align*}
which gives the desired estimate.

The divergence case is similar.
This finishes the proof of the first assertion.

For the assertion (ii),
we find $r \in (0,1)$ such that $f = 0$ and $|b_i|, |c_i| \le 1$ on $\{X \in Q: d(X) \le r\}$.
Then we take $\beta_1:= \log_{\delta_0} \beta_0$, where $\beta_0 = \beta_0(n,\nu,\theta_0/2,1,1)$ and $\delta_0 = \delta_0(n,\theta_0)$ are the constants from Remark \ref{GLremark}.
We observe that by Remark \ref{GLremark} and an iteration argument, we have
$$
|u(X)|\le N(d(X))^{\beta_1}\quad \forall X\in \bar Q.
$$
Thus, for any $\beta\in (0,\beta_1)$
and $X_0\in \partial_p Q$,
$$
\lim_{Q\ni X\to X_0}d^{-\beta}u(X)=0,
$$
and $d^{-\beta}u\in C(\bar Q)$.  Then we argue as in the
proof of the first assertion with $Q$ in place of $Q'$ with the smaller $\beta_1$ between the above $\beta_1$ and the one from the proof of assertion (i). The lemma is proved.
\end{proof}

The following 1-D example suggests that the growth condition on $b_i$
and $c_i$  near the boundary in the proposition above is in some sense optimal even in the elliptic case.
\begin{ex}
We take a nonnegative smooth  function $\eta$ on $\bR$ such that $\eta(x)=1$ for $x\le 1/3$ and $\eta=0$ for $x\ge 1/2$.
Consider $u(x) := - \eta(x)(\ln x)^{-1}$. Clearly, $u\in C([0,1/2])\cap C^\infty((0,1/2))$ and $u(0)=u(1/2)=0$. By a direct computation, it is easily seen that $u$ satisfies
\[
u'' + bu'=f\quad \text{in}\,\, (0,1/2),
\]
where $ b(x):= x^{-1} \left( 1 + 2(\ln x)^{-1} \right)  $ for $x\in (0,1/3)$ and $f\in L_\infty([0,1])$.
Note that $|b| \sim  x^{-1}$ near 0 and $\lim_{x\searrow 0}d^{-\beta} u=\infty$ for any $\beta>0$.
\end{ex}

\section{Proofs of the main theorems}   \label{proofs}

Now we are ready to prove the main results of the paper, Theorem~\ref{main}.
\begin{proof}[Proof of Theorem~\ref{main}]
We first prove the non-divergence case. Recall the notation introduced at the beginning of the previous section. For $k=1,2,\ldots$, denote
$$
Q_k=\{X\in Q\,|\,\dist(X,\partial_p Q)>1/k\},
$$
$L^k=L^{1,1/k}$, and $f^k=f^{1,1/k}$.
We shall find a unique solution $u_k$ to the equation
\begin{equation}
						\label{eq1023_1}
L^k u_k =f^k
\quad \text{in} \,\, Q,
\quad
u_k|_{\partial_p Q} = 0
\end{equation}
using Perron's method.
Since the $W^{2,1}_p$-solvability in cylindrical domains for parabolic operators with VMO coefficients is available (cf. \cite{DK09}) and $f^k$ are bounded in $L_{p_0} (Q)$,
in order to find a unique solution $u_k\in W(Q)$ to the equation \eqref{eq1023_1} using Perron's method, it suffices to construct a barrier function $\psi_k\in W(Q)$ satisfying
\begin{equation}
                                        \label{eq12.48}
L^k\psi_k\ge 1\quad\text{in}\,\,Q,\quad \psi_k=0\quad\text{on}\,\,\partial_p Q.
\end{equation}
For this purpose, we use an idea in \cite{Kr89}. Let $w\in W(Q)$ be the solution to the heat equation
\begin{equation}
                                            \label{eq11.20}
w_t-\Delta w=1\quad \text{in}\,\,Q,\quad
w=0\quad \text{on} \,\,\partial_p Q.
\end{equation}
The existence of such $w$ is due to \cite{EG82} and the fact that any domain satisfying the exterior measure condition also satisfies Wiener's criterion. See Appendix \ref{(A)domain}.
By the maximum principle, $w$ is strictly positive in $Q$. Denote $R=\diam(Q)$ and without loss of generality, we assume $Q\subset C_R$. Let $v=\cosh(\mu R)-\cosh(\mu|x|)$. Then $v\ge 0$ in $Q$, and a straightforward calculation shows that
\begin{equation}
                            \label{eq12.46}
L^k v\ge 1\quad\text{in}\,\,Q
\end{equation}
for $\mu$ sufficiently large. Denote $F(x,y):=\min\{x,y\}$ and let $F^{(\varepsilon)}$ be the standard mollification of $F$. Clearly, for any $\varepsilon>0$, $F^{(\varepsilon)}$ is a smooth function and
\begin{equation}
                                \label{eq12.45}
D_xF^{(\varepsilon)}\ge 0, \quad D_yF^{(\varepsilon)}\ge 0,\quad
D_xF^{(\varepsilon)}+D_yF^{(\varepsilon)}=1,\quad D^2 F^{(\varepsilon)}\le 0.
\end{equation}
Now we choose $\lambda$ sufficiently large such that $\lambda w\ge v$ in $Q_{2k}$, and let
$$
\psi_k=F^{(\varepsilon)}((1+\lambda)w,v).
$$
Then using the definition of $L^k$, \eqref{eq11.20}, \eqref{eq12.46}, and \eqref{eq12.45}, it is easily seen that $\psi_k$ satisfies \eqref{eq12.48} if $\varepsilon$ is sufficiently small.

By Proposition~\ref{UF} (ii),
there exists $N=N(n,\nu, \theta_0, \gamma, \diam(Q), \beta, p_0, \omega_{a_k}^*)
$ such that, for any $\beta \in (0, \beta_1)$,
\begin{equation*}
 \sup_{Q} d^{-\beta} u_k
\le N \|f\|_{ F(Q)},
\end{equation*}
where  $\beta_1 =\beta_1(n,\nu, \theta_0) \in (0,1]$ is the constant from
Proposition~\ref{UF}.
Note that by the definition of $L^k$ the choice of $\beta_1$ can be made independent of $k \in \bN$.
By using the simple equality
$$
\chi_2\big(\chi_1 A+(1-\chi_1)B\big)+(1-\chi_2)B=\chi_1\chi_2 A+(1-\chi_1\chi_2)B
$$
and the definition of $\omega_{a_k}^*$,
we see that $\omega_{a_k}^*$ has an upper bound
$$
\omega_{a_k}^*(r)\le N(\omega_a(r)+r/\varepsilon+r/\rho_0),
$$
which is independent of $k$. Here $\varepsilon$ and $\rho_0$ are constants from Proposition~\ref{UF}.
With $-u_k$ in place of $u_k$, we  obtain
$$
\inf_Q d^{-\beta} u_k \ge - N
\|f\|_{F(Q)}.
$$
Therefore, for any $X \in Q$, we have
\begin{equation}	\label{eq1001}
|u_k| \le N d^\beta \|f\|_F(Q),
\end{equation}
where $N$ is independent of $k$.

For an arbitrary parabolic cylinder $C \Subset_p Q$, thanks to the local $W^{1,2}_p$ estimate and the Sobolev embedding theorem, each $u_k$ is H\"older continuous in $C$ with a uniform H\"older norm.
Then by applying the Arzela--Ascoli theorem to the sequence $\{u_k\}$ on each $Q_m$, $m \in \bN$,
and using Cantor's diagonal argument, we find a subsequence, still denoted by $\{u_k\}$, converging locally uniformly to a function $u$ in $Q$.
Moreover, by the inequality \eqref{eq1001},
  $u \in C(\bar{Q})$ with $u = 0$ on $\partial_p Q$.
Now we show that $Lu = f$ in $C$.
One can find another parabolic cylinder $C'$ such that $C \Subset_p
 C' \Subset_p Q$.
Then for sufficiently
large $k$, we have
$$
L(u_{k+1}-u_k) = 0
$$
in $C'$.
By the  interior $W^{1,2}_p$ estimate, we have
$$
\|u_{k+1} - u_k\|_{W^{1,2}_{p_0}(C)} \le N \|u_{k+1} - u_k\|_{L_{p_0}(C')}.
$$
From this inequality and the fact that $u_k$ converges uniformly to $u$ in $C'$,
it follows that $u \in W_{p_0}^{1,2}(C)$ and $Lu = f$ in $C$.
Since $C \Subset_p Q$ is arbitrary, we have proved that $u \in W(Q)$ and $Lu = f$ in $Q$.
The uniqueness assertion follows immediately from the comparison principle, Corollary~\ref{comp}.

Now we deal with the divergence case.
Take the standard convolutions of coefficients and data functions $f$ with a non-negative mollifier, so that the mollifications $a_{ij}^{(k)}$, $b_i^{(k)}$, $c_i^{(k)}$,  $c_0^{(k)}$, $f_0^{(k)}$, and $f_i^{(k)}$ are in $C^\infty(\bar{Q}_{2k})$.
In particular, the mollified coefficients converge to their original ones locally in $L_2(Q)$. Define
$$
a_{ij}^k :=\big(a_{ij}^{(k)}\big)^{1,1/k},\quad
b^k_i:=\big(b_i^{(k)}\big)^{1,1/k},
$$
$$
c^k_i:=\big(c_i^{(k)}\big)^{1,1/k},
\quad
c^k_0:=\big(c_0^{(k)}\big)^{1,1/k},
$$
$$
f^k_0:=\big(f_0^{(k)}\big)^{1,1/k},
\quad
f^k_i:=\big(f_i^{(k)}\big)^{1,1/k},
$$
where each term is in $C^\infty(\bar{Q})$.
Let $L^k$ be the divergence type operator with the coefficients $a_{ij}^k$, $b_i^k$, $c_i^k$, and $c_0^k$.
Then we turn $L^k$ into a non-divergence type operator $\tilde{L}^k$ by defining
$$
\tilde{L}^k = \frac{\partial}{\partial t} - \tilde{a}_{ij}^k D_{ij} + \tilde{b}_i^kD_i + \tilde{c}_0^k,
$$
where
$$
\tilde{a}_{ij}^k = a_{ij}^k,
\quad
\tilde{b}_i^k = - D_j a_{ij}^k + b_i^k - c_i^k,
\quad
\tilde{c}_0^k = c_0^k - D_i c_i^k.
$$
Note that, due to the condition \eqref{D-sign_01}, we have $\tilde{c}_0^k \ge 0$. Indeed,
for any nonnegative function $\phi \in C_0^\infty(Q)$,
\begin{align*}
\int_Q \tilde{c}_0^k \phi \, dX
&= \int_Q \left(c_0^k \phi + c_i^k D_i \phi\right) \, dX\\
&= \int_Q \left(c_0^{(k)}\eta^{1/k} \phi +c_i^{(k)} D_i \left(\eta^{1/k}\phi \right) \right) \, dX \ge 0,
\end{align*}
which implies $\tilde{c}_0^k \ge 0$ on $Q$.
Then using the argument described above for the non-divergence case, we find a unique $u_k$ satisfying $\tilde{L}^k u_k = f_0^k - D_i f_i^k$ in $Q$ and $u_k|_{\partial_p Q} = 0$.
Since the coefficients and $f_0^k - D_i f_i^k$ are infinitely differentiable, $u_k$ is infinitely differentiable in $Q$ and continuous on $\bar{Q}$. Furthermore, $u_k$ also satisfies the divergence type equation $L^k u_k = f_0^k - D_i f_i^k$ in $Q$. That is,
\begin{multline}
							\label{eq1024_1}
\int_Q \left( - u_k D_t \phi  + a_{ij}^k   D_i u_k  D_j \phi+
   b_i^k D_i u_k \phi +  c_i^k u_k D_i \phi +c_0^k u_k \phi \right) \, dX
\\
= \int_Q \left(f_0^k \phi +  f_i^k D_i \phi \right)  \, d X
\end{multline}
for any $ \phi \in C_0^{\infty}(Q)$.
Similarly as in the non-divergence case above, with the De Giorgi--Nash--Moser estimate we find $u \in C(\bar{Q})$ with $u = 0$ on $\partial_p Q$, which is the local uniform limit of $\{u_k\}$.
Let us now prove that $u$ belongs to $H(Q)$ and satisfies the desired equation.
For cylinders $C \Subset_p C' \Subset_p Q$, by the standard local $L_2$-estimate,
$$
\|Du_k\|_{L_2(C)} \le N \|u_k\|_{L_2(C')} + N \|f_0^k\|_{L_{2(n+2)/(n+4)}(C')} + N \|f_i^k\|_{L_2(C')},
$$
where $N$ is independent of $k \in \bN$. Moreover, for sufficiently large $k$, the right-hand side is dominated by a quantity independent of $k$, so after taking a subsequence, $Du_k$ converges to $Du$ weakly. Then
$$
\int_Q a_{ij}^k D_i u_k D_j \phi \, dX \to \int_Q a_{ij} D_i u D_j \phi \, dX,
$$
where $\phi\in C_0^\infty(C)$.
More precisely,
$$
\left|\int_Q a_{ij}^k D_i u_k D_j \phi \, dX - \int_Q a_{ij} D_i u_k D_j \phi \, dX\right|
\le \|a_{ij}^k - a_{ij}\|_{L_2(C)}
\|Du_k D_j \phi\|_{L_2(C)} \to 0
$$
because $a_{ij}^k$ converges to $a_{ij}$ locally in $L_2$ and the $L_2(C)$-norm of $Du_k D_j \phi$ is uniformly bounded, and
$$
\int_Q a_{ij} D_i u_k D_j \phi \, dX \to
\int_Q a_{ij} D_i u D_j \phi \, dX
$$
because $D u_k$ converges to $Du$ weakly and $a_{ij} D_j \phi \in L_2(C)$.
The same reasoning is applied to the other terms
in \eqref{eq1024_1}.
Therefore,  by letting $k \to \infty$ in \eqref{eq1024_1}, we see that $u$ satisfies the divergence type equation $L u = f_0 - D_i f_i$. The fact that $u \in H(Q)$ follows from $Du \in L_{2,\text{loc}}(Q)$ and the equation.
The uniqueness follows from Corollary \ref{comp} as in the non-divergence case.
\end{proof}

\begin{rem}
For the non-divergence case, Assumption~\ref{RA}  is needed for the existence of the solution $u_k$
during the proof.
If one  can generalize the solvability in  smooth domains with more general coefficients (see, for instance, \cite{Kim08, DY10, Dg08}), our main results can also be improved.
\end{rem}

\begin{rem}
Following the proof of Theorem  \ref{main}, one can obtain a corresponding result for elliptic equations in a domain $\Omega$ in $\bR^n$ satisfying the exterior measure condition. The solution space is $W(\Omega):=W^2_{p_0}(\Omega)\cap C(\bar\Omega)$ in the non-divergence case for $p_0\in (n/2,\infty)$, and $W(\Omega):=W^1_{2}(\Omega)\cap C(\bar\Omega)$ in the divergence case. We leave the details to the interested reader.
\end{rem}

\appendix
\section{Regularity of (A) domains}
\label{(A)domain}
In this appendix, we will show that any (A) domain is regular, i.e., it satisfies Wiener's criterion.
We recall the definition of (A) domains:
\begin{defn} 
An open set $Q \subset \mathbb R  \sp {n+1}$ satisfies the condition (A)
if there exists a constant $\theta _{0}\in (0,1)$
such that for any $Y_0 \in \partial_p Q $ and $r>0$,
we have  $|C_{r}(Y_0)\setminus Q |>\theta _{0}|C_{r}|$.
\end{defn}
Now we introduce some standard notation and definitions.
One may consult, for instance,  \cite{Ev}.

Let $F$ be  the fundamental solution of the heat equation, namely,
$$
F(x,t) =
\left\{
\begin{aligned}
(4 \pi t)^{-n/2} \exp \left( - \frac{|x|^2}{4t} \right ), \quad t > 0,
\\
0, \quad t \le 0 ,
\end{aligned}
\right.
$$
and $E(x,t,r)$ be a heat ball of level set $r$ centered at $(x,t)$,
\begin{align*}
E(x,t,r):=\{ (y,s) \in \mathbb{R}^{n+1} \,|\,    F(x-y, t-s) \ge r \} .
\end{align*}
Now we list some useful properties of heat balls:
$$
E(r):= E(0,0,r) =  \{ (-x,-t) \in \bR^{n+1} | F(x,t) \ge r \},
$$
$$
F(r^{-1/n}x, r^{-2/n} t) = r F(x,t),
\quad
|E(r)| = r^{-1-2/n} |E(1)|<\infty,
$$
\begin{equation}
                            \label{eq12.12}
(y,s) \in E(1) \Leftrightarrow (x,t) \in E(r),
\quad
(x,t) = (r^{-1/n}y, r^{-2/n}s),
\end{equation}
$$
E(1) =  \Big\{(x,t) \in \bR^{n+1} | t < 0, (- 4\pi t)^{-n/2}
e^{\frac{|x|^2}{4t}}  \ge 1 \Big\} .
$$

\begin{defn} 
An open set $Q \subset \mathbb R  \sp {n+1}$ satisfies the condition (B)
if there exists a constant $\theta _{1}\in (0,1)$
such that, for any $Y_0 \in \partial_p Q $ and $\lambda >1$, there exists $k_0 \in \mathbb{N}$ satisfying
$$
| \{ (\lambda^{k_0})^{k+1} \ge F(Y_0 -Y) \ge (\lambda^{k_0})^k  \}  \setminus Q |
 >\theta _{1}   | \{ (\lambda^{k_0})^{k+1} \ge F(Y_0 -Y) \ge (\lambda^{k_0})^k  \}|
$$
for all $k \in \bN$.
\end{defn}

We will show in Theorem \ref{thmA5} that the condition (A)  implies the condition (B).

For each $ -\frac{1}{4 \pi} <  t <0 $, $(t,x) \in  E(1)$ if and only if
\begin{align*}
|x|^2 \le  2 n t \ln (-4 \pi t).
\end{align*}
This implies that in a neighborhood of the origin, the boundary of $E(1)$ can be represented as $t=\Phi(x)$, where $\Phi$ is a $C^2$ function and satisfies
\begin{equation}
                                    \label{eq4.21}
\Phi(0)=|D\Phi(0)|=|D^2\Phi(0)|=0.
\end{equation}

\begin{lemma}
\label{lem0510_1}
For any $\theta_0 >0$,
there exists $\theta = \theta(n, \theta_0) \in (0, 1 )$
 such that, for any $r > 0$,
\begin{equation}
							\label{eq0509_3}
\left| C_{\theta r^{-1/n}} \setminus E(r) \right| \le \frac{\theta_0}{4}
\left| C_{\theta r^{-1/n}} \right|.
\end{equation}
\end{lemma}

\begin{proof}
Due to the scaling property of $ x \to rx$, $ t \to r^2 x$
of the standard cylinder and the heat ball,
it suffices to prove \eqref{eq0509_3} when $r = 1$.
In this case, \eqref{eq0509_3} follows immediately from \eqref{eq4.21}.
\end{proof}

\begin{lemma}
          			\label{lem0510_2}
Let $\lambda > 1$ and $\theta \in (0,1)$ be fixed numbers.
Then one can choose a large $k_0 = k_0(n, \lambda, \theta, \theta_0) \in \bN$
such that, for any $k \in \bN$,
$$
\left| E \left( (\lambda^{k_0})^{k+1}\right) \right|
\le \frac{\theta_0}{4} \left| C_{\theta \lambda^{-k_0 k/n}} \right|.
$$
\end{lemma}

\begin{proof}
Set $r : = \lambda^{k_0}$.
Then
$$
\left| E \left( (\lambda^{k_0})^{k+1} \right) \right|
= \left| E \left( r^{k+1} \right) \right|
= r^{(k+1)(-1-2/n)} |E(1)|
$$
and
$$
\left| C_{\theta \lambda^{-k_0 k/n}} \right| =
\left| C_{\theta r^{-k/n}} \right| = \left(\theta r^{-k/n}\right)^{n+2}|C_1|.
$$
Thus,
$$
\frac{\left| E \left( r^{k+1} \right) \right|}{\left| C_{\theta r^{-k/n}} \right|}
= \frac{r^{-1-2/n} | E(1)|}{\theta^{n+2} |C_1|}
= \frac{1}{(\lambda^{k_0})^{1+2/n}} \, \frac{|E(1)|}{\theta^{n+2} |C_1|},
$$
which can be made small if $\lambda > 1$ and $k_0$ is sufficiently large.
The lemma is proved.
\end{proof}

\begin{theorem}[Condition (A) implies Condition (B)]            \label{thmA5}
Let $Q$ be  a (A)-domain in $\bR^{n+1}$. For any $Y_0 \in \partial_p Q$, and
$\lambda > 1$, there exist $k_0 = k_0(n, \lambda, \theta_0) \in \bN$
and  $\theta_1 = \theta_1(n, \theta_0) \in (0,1)$ such that
$$
\left| \left\{ (\lambda^{k_0})^{k+1} \ge F(Y_0-Y)
\ge (\lambda^{k_0})^k \right\} \setminus Q \right|
\ge \theta_1 \left| \left\{ (\lambda^{k_0})^{k+1}
\ge F(Y_0-Y) \ge (\lambda^{k_0})^k \right\} \right|
$$
for all $k \in \bN$.
\end{theorem}

\begin{proof}
Without loss of generality,  we assume that $Y_0 = (0,0) \in \partial_p Q$.
For a given $\lambda > 1$, take $k_0$ from Lemma \ref{lem0510_2},
where $\theta=\theta(n,\theta_0)$ is a number from Lemma \ref{lem0510_1}.
Set $r = \lambda^{k_0}$. Then
$$
\left\{ (\lambda^{k_0})^{k+1} > F(-Y) \ge (\lambda^{k_0})^k \right\}
= \left\{ r^{k+1} > F(-Y) \ge r^k \right\}
= E(r^k) \setminus E(r^{k+1}).
$$
Thus
$$
\left\{ (\lambda^{k_0})^{k+1} > F(-Y) \ge (\lambda^{k_0})^k \right\} \setminus Q
= \left[E(r^k) \setminus Q\right] \setminus \left[E(r^{k+1}) \setminus Q \right]
$$
and
\begin{multline}
							\label{eq0510_3}
\left| \left\{ (\lambda^{k_0})^{k+1} \ge F(-Y) \ge (\lambda^{k_0})^k \right\}
\setminus Q \right|
\ge \left|E(r^k) \setminus Q\right| - \left|E(r^{k+1}) \setminus Q \right|
\\
\ge \left|E(r^k) \setminus Q\right| - \left|E(r^{k+1}) \right|
\ge \left|E(r^k) \setminus Q\right| - \frac{\theta_0}{4} \left| C_{\theta r^{-k/n}}
\right|,
\end{multline}
where the last inequality is due to Lemma \ref{lem0510_2}.

On the other hand,
$$
\left|C_{\theta r^{-k/n}} \setminus Q\right|
= \left| \left(C_{\theta r^{-k/n}} \cap E(r^k) \right) \setminus Q \right| + \left| \left(C_{\theta r^{-k/n}}\setminus E(r^k) \right) \setminus Q \right|
$$
$$
\le \left| E(r^k) \setminus Q \right|
+ \left| C_{\theta r^{-k/n}}\setminus E(r^k) \right|
\le \left| E(r^k) \setminus Q \right|
+ \frac{\theta_0}{4} \left| C_{\theta r^{-k/n}} \right|,
$$
where the last inequality is due to Lemma \ref{lem0510_1}.
Along with \eqref{eq0510_3} and Definition \ref{A},
\begin{align*}
&\left| \left\{ (\lambda^{k_0})^{k+1} \ge F(-Y) \ge (\lambda^{k_0})^k \right\}
\setminus Q \right|
\ge \left|C_{\theta r^{-k/n}} \setminus Q\right| - \frac{\theta_0}{2}
\left| C_{\theta r^{-k/n}} \right|\\
&\,\ge \frac{\theta_0}{2} \left| C_{\theta r^{-k/n}} \right|
= \frac{\theta_0}{2} \,
\frac{\left| C_{\theta r^{-k/n}} \right|}{\left|E(r^k)\right|}
\, \left|E(r^k)\right|
= \frac{\theta_0  \theta^{n+2} }{2} \,  \frac{|C(1)|}{|E(1)|}
\, \left|E(r^k)\right|\\
&\,\ge \theta_1 \left| \left\{ (\lambda^{k_0})^{k+1} \ge F(Y_0-Y)
\ge (\lambda^{k_0})^k \right\} \right|,
\end{align*}
where
$$
\theta_1 = \frac{\theta_0 \theta^{n+2} }{2} \, \frac{|C(1)|}{|E(1)|}.
$$
That is, $\theta_1$ depends only on $n$ and $\theta_0$.
\end{proof}

We denote
$$
V_2( \mathbb{R}^{n+1} )=\{u\,|\,\nabla u\in L_2(\bR^{n+1}),\|u(\cdot,t)\|_{L_2(\bR^n)}\in L_\infty(\bR)\}.
$$
For a compact set $K$ in $\mathbb{R}^{n+1}$, recall the thermal capacity:
$$
\text{cap} (K)=\sup\{\mu(\bR^{n+1})\,|\,\mu\in M(K),F* \mu\le 1\},
$$
where $M(K)$ is the set of all nonnegative Radon measure supported in $K$, and the parabolic capacity:
\begin{align*}
\Gamma (K) = \inf \left\{  \sup_t \int_{\mathbb{R}^n}  u^2 (x,t)\,dx +
\int_{ \mathbb{R} } \int_{  \mathbb{R}^n } |\nabla u|^2 \,dX
\right\} ,
\end{align*}
where the function $u$ is taken over all functions in $V_2 ( \mathbb{R}^{n+1} )$
with compact support such that  $K \subset \text{int} \{ X : u(X) \ge 1 \}$.

The following result can be found in \cite{GZ}.

\begin{lemma}
                        \label{lemA6}
For any compact set $K$ in $\mathbb{R}^{n+1}$,
\begin{align*}
\text{cap} (K) \ge \frac{1}{2} \Gamma (K).
\end{align*}
\end{lemma}

As a consequence, we have
\begin{lemma}
                            \label{lemA7}
For any  set $K$   in $\mathbb{R}^{n+1}$,
we have, for some $N>0$,
\begin{align*}
\text{cap} (K) \ge N |K |^{\frac{n}{n+2}} .
\end{align*}
\end{lemma}

\begin{proof}
From Lemma \ref{lemA6},
\begin{align*}
\text{cap} (K) \ge \frac{1}{2} \Gamma (K) .
\end{align*}
By the parabolic type Gagliardo--Nirenberg--Sobolev inequality (see, for instance, \cite[Theorem IV.6.9]{Li}), for any  $u \in V_2 ( \mathbb{R}^{n+1} )$
with compact support such that  $K \subset \text{int} \{ X : u(X) \ge 1 \}$,
\begin{equation*}
\sup_t \int_{\mathbb{R}^n}  u^2 (x,t)\,dx +
\int_{ \mathbb{R} } \int_{  \mathbb{R}^n } |\nabla u|^2 \,dX
\ge N\left( \int   |u|^{ \frac{2(n+2)}{n} }   \,dX \right)^{\frac n {n+2}}
\ge N |K |^{\frac{n}{n+2}}.
\end{equation*}
By the definition of $\Gamma(K)$, the lemma follows.
\end{proof}

Finally, we show that the condition (B) implies Wiener's criterion.

\begin{thm}
Let $Q$ satisfy the condition (B). Then any point on $\partial_p Q$ is regular.
Namely,
$X_0 \in \partial_p Q$ satisfies the following Wiener's  criterion from \cite[Theorem 1]{EG82}:
for any $\lambda>1$,
\begin{align*}
\sum_{k=1}^{\infty} \lambda^k
\text{cap} (Q^c  \cap \{  \lambda^{k+1} \ge F(X_0 - X) \ge \lambda^k \}  ) = \infty.
\end{align*}
\end{thm}

\begin{proof}
We take the constant $k_0$ from Theorem \ref{thmA5}. It then follows from Theorem \ref{thmA5}, Lemma \ref{lemA7}, and \eqref{eq12.12} that
\begin{align*}
& \sum_{k=1}^{\infty} \lambda^k
\text{cap} (Q^c  \cap \{  \lambda^{k+1} \ge F(X_0 - X) \ge \lambda^k \}  )
\\
& \ge  \frac{1}{\lambda^{k_0}}
\sum_{k=1}^{\infty} \lambda^{k_0k}
\text{cap} (Q^c  \cap \{  \lambda^{k_0(k+1)} \ge F(X_0 - X) \ge \lambda^{k_0k} \}  )
\\
& \ge  \frac{N}{\lambda^{k_0}}
\sum_{k=1}^{\infty} \lambda^{k_0k}  |Q^c  \cap \{  \lambda^{k_0(k+1)} \ge F(X_0 - X) \ge \lambda^{k_0k} \}  |^{\frac n {n+2}}\\
& \ge  \frac{N}{\lambda^{k_0}}
\sum_{k=1}^{\infty} \lambda^{k_0k}  \theta_1^{\frac n {n+2}}|\{  \lambda^{k_0(k+1)} \ge F(X_0 - X) \ge \lambda^{k_0k} \}  |^{\frac n {n+2}}\\
&=  \frac{N}{\lambda^{k_0}}\sum_{k=1}^{\infty} \theta_1^{\frac n {n+2}}|\{  \lambda^{k_0} \ge F(X_0 - X) \ge 1 \}  |^{\frac n {n+2}}= \infty .
\end{align*}
The theorem is proved.
\end{proof}

\end{document}